\documentclass[a4paper,11pt]{article}
\usepackage[english]{babel}
\usepackage{amssymb}
\usepackage{amsmath}
\usepackage{amsthm}
\usepackage[dvips]{graphicx}
\usepackage{color}
\usepackage{hyperref}

\newtheorem{cor}{Corollary}[section]
\newtheorem{defi}[cor]{Definition}
\newtheorem{prop}[cor]{Proposition}

\newtheorem{teo}[cor]{Theorem}

\newtheorem{rem}[cor]{Remark}
\newtheorem{prob}[cor]{Problem}
\newtheorem{ass}[cor]{Assumption}

\def\R{\mathbb{R}}
\def\mbN{\mathbb{N}}
\def\pd{\partial}
\def\eoc{{\rm EOC}}
\def\eocs{{\rm EOCs}}
\def\triod{\mathcal{T}}
\def\domain{\Omega}
\def\interpol{I_{h}}
\def\dim{{\rm dim}}
\def\cond{{\rm cond}}
\def\diag{{\rm diag}}
\numberwithin{equation}{section}
\title{A converging finite element scheme for motion by curvature of a network with a triple junction}

\author{Paola Pozzi 
\thanks{Fakult\"at f\"ur Mathematik, Universit\"at Duisburg-Essen, 
Thea-Leymann-Stra\ss e 9, 
45127 Essen, Germany, 
\url{paola.pozzi@uni-due.de}} 
\and 
and Bj\"orn Stinner 
\thanks{Mathematics Institute, 
Zeeman Building, 
University of Warwick, 
Coventry CV4 7AL, 
United Kingdom, 
\url{bjorn.stinner@warwick.ac.uk}} 
}

\begin{document}

\maketitle

\begin{abstract}
 A new semi-discrete finite element scheme for the evolution of three parametrised curves by curvature flow that are connected by a triple junction is presented and analysed. In this triple junction, conditions are imposed on the angles at which the curves meet. One of the key problems in analysing motion of networks by curvature law is the choice of a tangential velocity that allows for motion of the triple junction, does not lead to mesh degeneration, and is amenable to an error analysis. Our approach consists in considering a perturbation of a classical smooth formulation. The problem we propose admits a natural variational formulation that can be discretized with finite elements. The  perturbation can be made arbitrarily small when a regularisation parameter shrinks to zero. Convergence of the new scheme including optimal error estimates are proved. These results are supported by some numerical tests. Finally, the influence of the small regularisation parameter on the properties of scheme and the accuracy of the results is numerically investigated. 
\end{abstract}

\bigskip

%65M12  	Stability and convergence of numerical methods
% 65M15  	Error bounds
%65M60  	Finite elements, Rayleigh-Ritz and Galerkin methods, finite methods
%65M20  	Method of lines

\bigskip

 
%%%%%%%%%%%%%%%%%%%%%%%%%%%%%%%%%%%%%%%%%%%
\section{Introduction}

We numerically study the planar evolution of networks formed by curves that move by curvature flow and that meet in triple junctions at prescribed angles. Such or related problems occur in applications in materials science (evolution of grain boundaries between crystalline phases, for instance, see \cite{TayCahHan1992}) or in fluids (equilibria in multi-phase flow, for instance, see \cite{Cah1977}). The focus here is on evolving \emph{triods} formed by three curves, each with one fixed end point and connected to a mobile triple junction with the other end point. Curvature flow refers to the law that the normal velocity $V^{(i)}$ of each curve in its (unit) normal direction $\nu^{(i)}$ coincides with its curvature $H^{(i)}$ with respect to the orientation defined by the unit normal, 
\begin{equation} \label{eq:curvflow}
 V^{(i)} = H^{(i)}, \quad i=1,2,3. 
\end{equation}
In the triple junction, the condition
\begin{equation} \label{eq:forcebal}
 \sum_{i=1}^{3} \tau^{(i)} = 0
\end{equation}
is imposed, where $\tau^{(i)}$ is the unit tangent vector of curve $i$ pointing into the curve. This condition can be interpreted as a force balance and is known as \emph{Herring's condition} in materials science \cite{Her1951,PriYu1994} and as \emph{Young's law} in fluids \cite{You1805,BroGarSto1998}. Here, it implies that the curves form angles of $120^\circ$ at the triple junction. 

Curvature flow is driven by the length functional. Denoting a regular parametrisation of a curve by $\tilde{u} : [0,1] \to \R^2$, this functional reads 
\[
 \tilde{E}(\tilde{u}) = \int_{0}^{1} | \tilde{u}_{x} | dx. 
\]
The curve may now be deformed in any direction $\tilde{\phi} : [0,1] \to \R^2$. The variation of the functional in this direction is 
\begin{multline} \label{eq:var_functional}
 \langle \tilde{E}'(\tilde{u}), \tilde{\phi} \rangle = \int_{0}^{1} \frac{\tilde{u}_{x}}{| \tilde{u}_{x} |} \cdot \tilde{\phi}_{x} dx \\
 = \frac{\tilde{u}_{x}}{| \tilde{u}_{x} |} \cdot \tilde{\phi} \Big{|}_{0}^{1} - \int_{0}^{1} \frac{1}{| \tilde{u}_{x} |} \Big{(} \frac{\tilde{u}_{x}}{| \tilde{u}_{x} |} \Big{)}_{x} \cdot \tilde{\phi} \,| \tilde{u}_{x} | dx = \tilde{\tau} \cdot \tilde{\phi} \Big{|}_{0}^{1} - \int_{0}^{1} \tilde{\kappa} \cdot \tilde{\phi} \,| \tilde{u}_{x} | dx, 
\end{multline}
where $\tilde{\tau} = \tilde{u}_{x} / | \tilde{u}_{x} |$ is a unit tangent field and $\tilde{\kappa} = (\tilde{u}_{x} / | \tilde{u}_{x} |)_{x} / | \tilde{u}_{x} |= \tilde{H} \tilde{\nu}$ is the curvature vector. Curvature flow \eqref{eq:curvflow} can be formulated as the gradient flow with respect to the $L^2$~inner product on the curve, which here is the $L^2$ inner product on the reference domain with weighting $|\tilde{u}_{x}|$. For a closed curve, \eqref{eq:var_functional} yields the variational formulation 
\begin{equation} \label{eq:gradflow_Dzi}
 \int_{0}^{1} \tilde{u}_{t} \cdot \tilde{\phi} | \tilde{u}_{x} | + \frac{\tilde{u}_{x}}{| \tilde{u}_{x} |} \cdot \tilde{\phi}_{x} dx = 0. 
\end{equation}
The velocity of the parametrisation then satisfies 
\begin{equation} \label{eq:csf}
 \tilde{u}_{t} = \frac{1}{| \tilde{u}_{x} |} \Big{(} \frac{\tilde{u}_{x}}{| \tilde{u}_{x} |} \Big{)}_{x} = \tilde{\kappa} 
\end{equation}
and is in purely in normal direction, i.e., it realises the geometric evolution \eqref{eq:curvflow} without any tangential velocity contributions. 

At first view, the variation \eqref{eq:var_functional} also looks attractive for the triod case. Summing up the boundary terms for three curves yields the angle condition \eqref{eq:forcebal}, which then naturally is satisfied in a variational formulation obtained by summing up \eqref{eq:gradflow_Dzi}. However, the purely normal velocity implies that the triple junction then is immobile. In fact, if moved in the normal direction with respect to one of the curves then, due to the angle condition \eqref{eq:forcebal}, the movement would involve tangential components with respect to the other two curves, but which is incompatible with \eqref{eq:csf}. 

Analytical studies of networks thus resort to parametrisations that realise \eqref{eq:curvflow} but also allow for tangential velocity components. A popular choice is \cite{ManNovPluSch2016}
\begin{equation} \label{eq:csf2}
 \tilde{u}_{t} | \tilde{u}_{x} |^2 = \tilde{u}_{xx}. 
\end{equation}
This is a gradient flow of the Dirichlet energy $\int | \tilde{u}_{x} |^2 / 2 dx$ with respect to an $L^2$ inner product with weighting $| \tilde{u}_{x} |^2$. It can be interpreted as a reparametrisation of the curves by solving a harmonic map flow for the tangential movement, see \cite{EllFri2016} for a presentation and in-depth analysis of the procedure. 
The analytical study of networks moving according to \eqref{eq:csf2} is treated for instance in the survey \cite{ManNovPluSch2016}, where questions such as existence, uniqueness, singularity formation and behaviour of the flow are discussed in detail. 
It turns out that this idea is also beneficial for numerical simulations. 

But let us first get back to \eqref{eq:csf}. Based on the variational formulation \eqref{eq:gradflow_Dzi}, a linear finite element scheme was proposed in \cite{Dzi1991} (and, thanks to an intrinsic formulation on evolving triangulations, even for closed surfaces). Convergence was proved for the semi-discrete scheme for curves in \cite{Dzi1994} where the key challenge was to control the length element $| \tilde{u}_{x} |$. The scheme mimics the geometric evolution in that also the vertices, i.e., the images of the mesh nodes on $[0,1]$ under the piecewise linear finite element solution, move approximately in normal direction. In the long term, in general, the length element will thus evolve strong discrepancies. Vertices will accumulate in some places while, elsewhere, segments between vertices may be stretched. Whilst this might be acceptable to some extent for simulations of closed curves, redistribution of the vertices in tangential direction is mandatory in the case of triods for the same reasons as in the continuum case, namely, to compensate for tangential movements of the triple junction. 

The idea of using \eqref{eq:csf2} instead to simulate curves forming networks was picked up in \cite{BroWet1993}. Finite difference techniques were used for the PDE and the triple junction condition \eqref{eq:forcebal}. Whilst the schemes behaved well in practice, convergence was investigated numerically only. In \cite{DecDzi1994} a finite element method based on \eqref{eq:csf2} for closed curves was presented. Convergence of the semi-discrete scheme was proved using a fixed point argument. But using \eqref{eq:csf2} to develop a finite element scheme for a triod is not straightforward because of the angle condition \eqref{eq:forcebal}. In fact, if three curves $\tilde{u}^{(i)}$, $i=1,2,3$, move by \eqref{eq:csf2} whilst forming a triple junction then a natural boundary condition in that triple junction reads $\sum_{i} \tilde{u}^{(i)}_{x} = 0$ rather than \eqref{eq:forcebal}, which can be written as $\sum_{i} \tilde{u}^{(i)}_{x} / |\tilde{u}^{(i)}_{x}| = 0$.

The idea of our approach is to use \eqref{eq:csf} for the movement in normal direction and to realise the triple junction condition, and then to combine it with \eqref{eq:csf2} scaled with a small parameter $\epsilon>0$ for some tangential movement, where the scaling serves to ensure that the impact on the geometric evolution and the triple junction condition is small. More precisely, instead of $\tilde{u}_{t}$ we consider $(\tilde{u}_{t} \cdot \tilde{\nu}) \tilde{\nu}$ in \eqref{eq:gradflow_Dzi} and $(\tilde{u}_{t} \cdot \tilde{\tau}) \tilde{\tau}$ in \eqref{eq:csf2}. Formulating the latter weakly and accounting for the scaling with $\epsilon>0$, the weak formulation for a single curve then reads
 \begin{equation} \label{eq:var_single}
 \int_{\domain} \big{(} (\tilde{u}_{t} \cdot \tilde{\nu}) (\tilde{\nu} \cdot \tilde{\varphi}) |\tilde{u}_{x}| + \tilde{\tau} \cdot \tilde{\varphi}_{x} \big{)} dx + \epsilon \int_{\domain} \big{(} (\tilde{u}_{t} \cdot \tilde{\tau}) (\tilde{\tau} \cdot \tilde{\varphi}) |\tilde{u}_{x}|^{2} + \tilde{u}_{x} \cdot \tilde{\varphi}_{x} \big{)} dx = 0. 
 \end{equation}
 This can now be extended to three curves forming a triod. See Problem \ref{prob_cont} for a complete formulation of the variational problem including initial and boundary conditions, which is at the centre of our numerical approach. Regarding the corresponding strong problem we refer to \eqref{PDEu} for the evolution law of the curves and to \eqref{JPC} for the triple junction condition. Observe that the curves satisfy \eqref{eq:curvflow} and \eqref{eq:forcebal} up to terms scaling with $\epsilon$.

Variational problems of a form similar to \eqref{eq:var_single} are amenable to a discretisation with piecewise linear conforming finite elements as $\tilde{\tau}$ and $\tilde{\nu}$ involve first spatial derivatives of $\tilde{u}$ only. Our \emph{main result} is a convergence proof of the thus obtained semi-discrete finite element scheme. In Theorem \ref{th_conv} we show linear and, thus, optimal convergence of the error in the $L^\infty(L^2)$ norm of the first spatial derivative and in the $L^2(L^2)$ norm of the velocity.

For the proof the procedure in \cite{DecDzi1994} was followed, where convergence of a semi-discrete finite element scheme for \eqref{eq:csf2} is shown in the case of a single closed curve. A fixed point map is constructed and analysed that satisfies a desired error estimate. It benefits from the linearity of the second-order spatial differential operator (diffusion term) in \eqref{eq:csf2}. The non-linearity of the diffusion term in \eqref{eq:gradflow_Dzi} and \eqref{eq:var_single} required significant adaptations from our part. Further extensions of the arguments were due to the splitting of the velocity into a normal and a tangential part. Our error estimates depend in an unfavourable way on $\epsilon$, the generic constants scale with $\epsilon^{-1}$. We found that, in practice, the method works well for small values on coarse meshes. The impact of the $\epsilon$ was quantitatively assessed in numerical simulations. We report on numerical convergence results as $\epsilon \to 0$ and on the conditioning of the system matrix.

Harmonic maps to ensure a good distribution of vertices also underpin the ideas in \cite{BarGarNue2007,BarGarNue2011}. Their fully discrete schemes generally have good stability properties and variationally satisfy the triple junction condition, whilst convergence hasn't been proved yet. For other, more recent computational approaches and ideas centred around goal-oriented r-adaptivity for geometric evolution problems of single curves or surfaces we refer to \cite{BalMik2011,MikRemSarSev2014,MacNolRowIns2019}. But we are not aware of any work that addresses convergence of schemes (in a parametric setting) for evolving networks with triple junctions subject to \eqref{eq:curvflow} and \eqref{eq:forcebal}, our result for a semi-discrete scheme seems the first. For completeness, let us also briefly mention interface capturing approaches that avoid the need of look after the mesh quality. Such approaches comprise phase field models \cite{BroRei1993,BreMas2017} and level set methods \cite{MerBenOsh1994,SmiSolCho2002}, for an overview we refer to \cite{DecDziEll2005}. 

In the following section we precisely define evolving triods and formulate the continuum problem that we intend to approximate, and we also clarify the requirements on the solution. After, we present the finite-element scheme. Section \ref{sec:convergence} contains the convergence analysis and the main result, Theorem \ref{th_conv}. In the last section we discretise in time and report on several numerical tests that corroborate our theoretical findings. We also report on the influence of the small parameter $\epsilon$ and display the effectiveness of the scheme for challenging initial data. 

\bigskip
\noindent \textbf{Acknowledgements:} This project was supported by the Deutsche Forschungsgemeinschaft (DFG, German Research Foundation), Projektnummer 404870139, and by the Engineering and Physical Sciences Research Council (EPSRC, United Kingdom), grant no EP/K032208/1. The second author would like to thank the Isaac Newton Institute for Mathematical Sciences, Cambridge, for support and hospitality during the programme \emph{Geometry, compatibility and structure preservation in computational differential equations}, where work on this paper was undertaken.

%%%%%%%%%%%%%%%%%%%%%%%%%%%%%%%%%%%%%%%%%%%
\section{Continuum problem, triod evolution}

In the formulation of the problem we omit any in-depth discussion of the appropriate function spaces. Typically, one would show short-time existence by applying Solonnikov theory and a fixed point argument in parabolic H\"older spaces, see for instance \cite{BroRei1993}, \cite{ManNovPluSch2016}. Since this is outside of the scope of this paper we henceforth assume the existence of a sufficiently smooth solution on some time interval. In Assumption~\ref{ass_cont_sol} below we list the regularity assumptions that we need for the error analysis. 

\begin{defi} \label{def:triod}
Given three fixed points $P_{i} \in \R^{2}$, $i=1,2,3$, a triod is the union of three curves $u^{(i)}: \domain \to \R^{2}$, $\domain = [0,1]$ connecting a joint starting point with the points $P_{i}$. More precisely, we denote this set of triods by 
\begin{align*}
 \triod_{P} := \{ \, \Gamma = (u^{(1)},u^{(2)},u^{(3)}) \, | \, & u^{(i)} \in W^{1,2}(\domain,\R^{2}) \text{ regular almost everywhere,} \\
 & u^{(i)}(1) = P_{i}, \quad i=1,2,3, \\
 & u^{(1)}(0) = u^{(2)}(0) = u^{(3)}(0) \, \}.
\end{align*}
For some small positive $\epsilon \leq \tfrac{1}{2}$, the energy associated with a triod $\Gamma \in \triod_{P}$ is given by 
\begin{align*}
 E(\Gamma) = \sum_{i=1}^{3} E_{\epsilon}(u^{(i)}), \quad \mbox{where } E_{\epsilon}(u^{(i)}) = \int_{\domain} \Big{(} |u^{(i)}_{x}| + \frac{\epsilon}{2} |u^{(i)}_{x}|^{2} \Big{)} dx.
\end{align*}
Given three triods $\Gamma = \{ u^{(1)},u^{(2)},u^{(3)} \}$, $\Upsilon = \{ v^{(1)},v^{(2)},v^{(3)} \}$, and $\Sigma = \{ w^{(1)},w^{(2)},w^{(3)} \} \in \triod_{P}$ we define 
\begin{align*}
 \langle \Upsilon, \Sigma \rangle_{\Gamma} := \sum_{i=1}^{3} \langle v^{(i)}, w^{(i)} \rangle_{u^{(i)}}
\end{align*}
where 
\begin{align}\label{innerprod}
 \langle v^{(i)}, w^{(i)} \rangle_{u^{(i)}} := \int_{\domain} \big{(} (v^{(i)} \cdot \nu^{(i)}) (w^{(i)} \cdot \nu^{(i)}) |u^{(i)}_{x}| + \epsilon (v^{(i)} \cdot \tau^{(i)}) (w^{(i)} \cdot \tau^{(i)}) |u^{(i)}_{x}|^{2} \big{)} dx
\end{align}
is a weighted $L^{2}$ inner product, and where we used the notation
\[
 \tau^{(i)} = \frac{u^{(i)}_{x}}{|u^{(i)}_{x}|} \qquad \mbox{ and } \qquad \nu^{(i)} = (\tau^{(i)})^\perp = \frac{(u^{(i)}_{x})^{\perp}}{|u^{(i)}_{x}|}.
\]
\end{defi}

\begin{figure}
 \begin{center}
 \includegraphics[width=7cm]{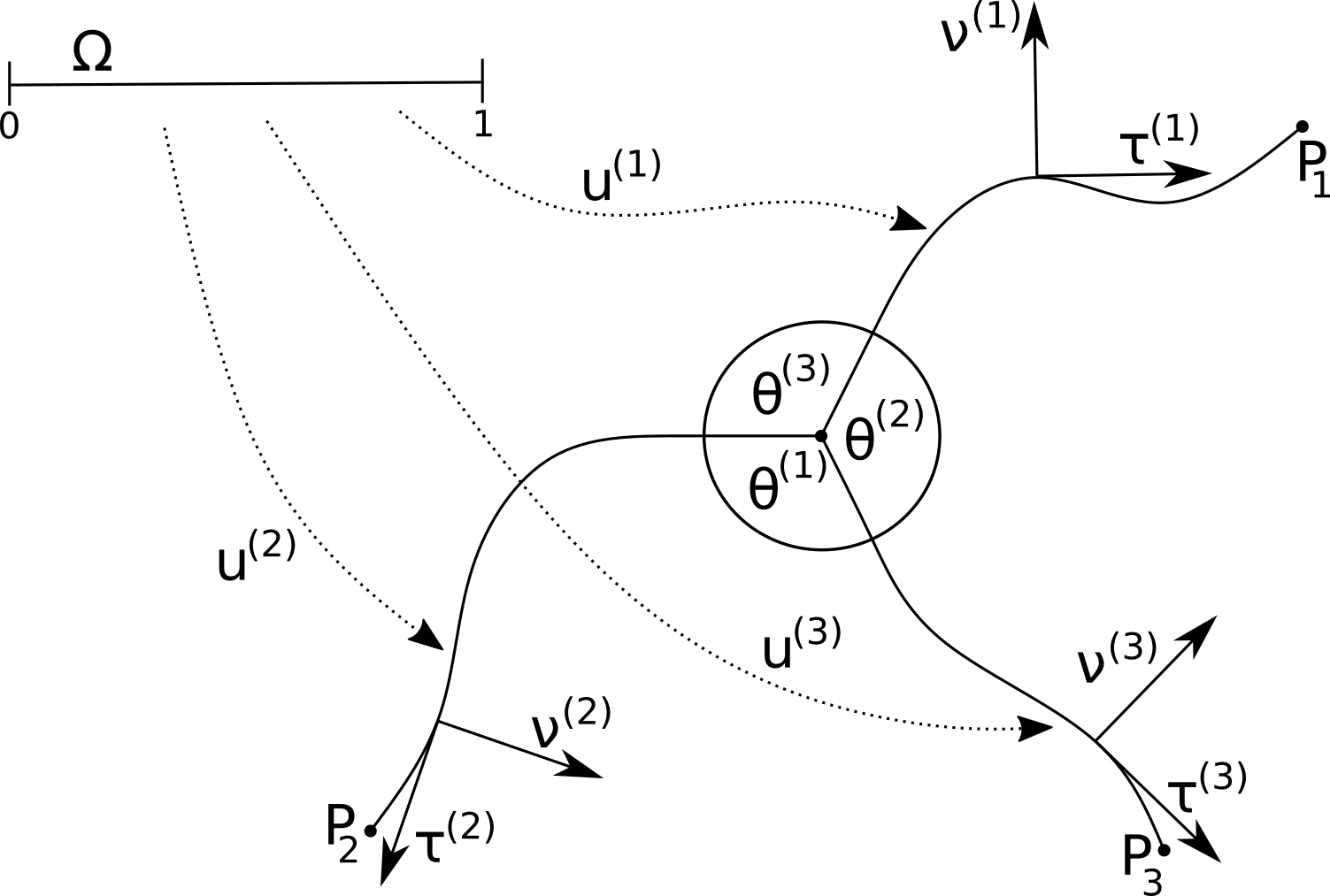} 
 \end{center}
 \caption{Illustration of a triod, see Definition \ref{def:triod} for the notation.}
 \label{fig:triod}
\end{figure}

See Figure \ref{fig:triod} for an illustration of a triod. Note that if $P_{1}=P_{2}=P_{3}$ then the triod is actually a so-called theta-network with a fixed point. When $P_{1}=P_{2}=P_{3}=0$ we write simply $\triod_{0}$. 

On the set of (sufficiently smooth) triods we consider the gradient flow dynamics 
\[
 \langle \pd_{t} \Gamma(t), \Phi \rangle_{\Gamma(t)} = - \langle E'(\Gamma(t)), \Phi \rangle \quad \forall \Phi \in \triod_{0}.
\]
Here, $E'(\Gamma(t))$ is the variation of the energy, i.e., writing $\Gamma(t) = (u^{(1)}(t),u^{(2)}(t),u^{(3)}(t)) \in \triod_{P}$, for any $\Phi = (\varphi^{(1)},\varphi^{(2)},\varphi^{(3)}) \in \triod_{0}$, 
\[
 \langle E'(\Gamma), \Phi \rangle = \sum_{i=1}^{3} \epsilon \int_{\domain} u_{x}^{(i)} \cdot \varphi_{x}^{(i)} dx + \int_{\domain} \frac{u_{x}^{(i)}}{|u_{x}^{(i)}|} \cdot \varphi_{x}^{(i)} dx.
\]
Note also that
\[
 \langle \pd_{t} \Gamma, \Phi \rangle_{\Gamma} = \sum_{i=1}^{3} \int_{\domain} (u_{t}^{(i)} \cdot \frac{(u_{x}^{(i)})^{\perp}}{|u_{x}^{(i)}|}) (\varphi^{(i)} \cdot \frac{(u_{x}^{(i)})^{\perp}}{|u_{x}^{(i)}|}) |u_{x}^{(i)}| + \epsilon (u_{t}^{(i)} \cdot \frac{u_{x}^{(i)}}{|u_{x}^{(i)}|}) (\varphi^{(i)} \cdot \frac{u_{x}^{(i)}}{|u_{x}^{(i)}|}) |u_{x}^{(i)}|^{2} dx.
\]

\begin{prob} \label{prob_cont}
 Given $\epsilon \in (0,\tfrac{1}{2}]$ and an initial triod $\Gamma_{0} = (u^{(1)}_{0},u^{(2)}_{0},u^{(3)}_{0}) \in \triod_{P}$ with points $P_{i} \in \R^{2}$, $i=1,2,3$, find a time interval $[0,T]$, $T \in (0,\infty)$, and a family of triods $\Gamma(t) = (u^{(1)}(t),u^{(2)}(t),u^{(3)}(t)) \in \triod_{P}$, $t \in [0,T]$, such that $\Gamma(0) = \Gamma_{0}$ and such that for almost every $t \in (0,T)$ and all $\Phi = (\varphi^{(1)},\varphi^{(2)},\varphi^{(3)}) \in \triod_{0}$ 
 \begin{multline} \label{eq:prob_weak}
 \sum_{i=1}^{3} \Bigg{(} \int_{\domain} (u_{t}^{(i)} \cdot \frac{(u_{x}^{(i)})^{\perp}}{|u_{x}^{(i)}|}) (\varphi^{(i)} \cdot \frac{(u_{x}^{(i)})^{\perp}}{|u_{x}^{(i)}|}) |u_{x}^{(i)}| dx + \epsilon \int_{\domain} (u_{t}^{(i)} \cdot \frac{u_{x}^{(i)}}{|u_{x}^{(i)}|}) (\varphi^{(i)} \cdot \frac{u_{x}^{(i)}}{|u_{x}^{(i)}|}) |u_{x}^{(i)}|^{2} dx \Bigg{)} \\
 = -\sum_{i=1}^{3} \Bigg{(} \epsilon \int_{\domain} u_{x}^{(i)} \cdot \varphi_{x}^{(i)} dx + \int_{\domain} \frac{u_{x}^{(i)}}{|u_{x}^{(i)}|} \cdot \varphi_{x}^{(i)} dx \Bigg{)}.
 \end{multline}
\end{prob}

The above gradient flow gives rise to an initial-boundary value problem for a system of PDEs. Let us denote the curvature vectors by $\kappa^{(i)} := \tau^{(i)}_{x} / |u_{x}^{(i)}|$, $i=1,2,3$. Observe that 
\begin{align*}
 (|u_{x}^{(i)}|)_{x} &= \frac{u_{xx}^{(i)} \cdot u_{x}^{(i)}}{|u_{x}^{(i)}|} = u_{xx}^{(i)} \cdot \tau^{(i)},\\
 \tau_{x}^{(i)} &= \Big{(} \frac{u_{x}^{(i)}}{|u_{x}^{(i)}|} \Big{)}_{x} = \frac{u_{xx}^{(i)}}{|u_{x}^{(i)}|} - \frac{u_{x}^{(i)} (u_{xx}^{(i)} \cdot \tau^{(i)})}{|u_{x}^{(i)}|^2} \\
 &= \frac{1}{|u_{x}^{(i)}|} \big{(} u_{xx}^{(i)} - (u_{xx}^{(i)} \cdot \tau^{(i)}) \tau^{(i)} \big{)} = \frac{1}{|u_{x}^{(i)}|} \big{(} u_{xx}^{(i)} \cdot \nu^{(i)} \big{)} \nu^{(i)}. 
\end{align*}
Partial integration on the right-hand-side of \eqref{eq:prob_weak} yields that 
\begin{align*}
 \sum_{i=1}^{3} & \int_{\domain} (u_{t}^{(i)} \cdot \nu^{(i)}) (\nu^{(i)} \cdot \varphi^{(i)}) |u_{x}^{(i)}| + \epsilon (u_{t}^{(i)} \cdot \tau^{(i)}) (\tau^{(i)} \cdot \varphi^{(i)}) |u_{x}^{(i)}|^{2} dx \\
 &= -\sum_{i=1}^{3} \int_{\domain} \epsilon u_{x}^{(i)} \cdot \varphi_{x}^{(i)} + \tau^{(i)} \cdot \varphi_{x}^{(i)} dx \displaybreak[0] \\
 &= -\sum_{i=1}^{3} \big{[} ( \epsilon u_{x}^{(i)} + \tau^{(i)} ) \varphi^{(i)} \big{]}_{0}^{1}
 + \sum_{i=1}^{3} \int_{\domain} \big{(} \epsilon (\tau^{(i)} |u_{x}^{(i)}|)_{x} + \tau_{x}^{(i)} \big{)} \cdot \varphi^{(i)} dx
 \displaybreak[0] \\
 &= \sum_{i=1}^{3} \big{(} \tau^{(i)}(0) + \epsilon u_{x}^{(i)}(0) \big{)} \varphi^{(i)}(0) - \big{(} \tau^{(i)}(1) + \epsilon u_{x}^{(i)}(1) \big{)} \varphi^{(i)}(1) \\
 & \quad + \sum_{i=1}^{3} \int_{\domain} (\epsilon |u_{x}^{(i)}| + 1) (\tau_{x}^{(i)} \cdot \nu^{(i)}) (\nu^{(i)} \cdot \varphi^{(i)}) + \epsilon (u_{xx}^{(i)} \cdot \tau^{(i)}) (\tau^{(i)} \cdot \varphi^{(i)}) dx. 
\end{align*}
Separating the normal from the tangential terms yields the following strong equations: 
\begin{align*}
 (u_{t}^{(i)} \cdot \nu^{(i)}) \nu^{(i)} |u_{x}^{(i)}| &= (1 + \epsilon |u_{x}^{(i)}|) \tau_{x}^{(i)} = \frac{1}{|u_{x}^{(i)}|} (\nu^{(i)} \cdot u_{xx}^{(i)}) \nu^{(i)} + \epsilon |u_{x}^{(i)}|^2 \kappa^{(i)},\\
 (u_{t}^{(i)} \cdot \tau^{(i)}) \tau^{(i)} |u_{x}^{(i)}|^{2} &= (\tau^{(i)} \cdot u_{xx}) \tau^{(i)}. 
\end{align*}
Using that $\varphi^{(i)}(1) = 0$ and that $\varphi^{(1)}(0) = \varphi^{(2)}(0) = \varphi^{(3)}(0)$ we furthermore deduce that 
\[
 \sum_{i=1}^{3} \big{(} \tau^{(i)}(0) + \epsilon u_{x}^{(i)}(0) \big{)} = 0.
\]
Thus, in its classical form the PDE problem is given by 

\begin{align}
u_{t}^{(i)} &= \frac{u_{xx}^{(i)}}{|u_{x}^{(i)}|^{2}} + \epsilon |u_{x}^{(i)}| \kappa^{(i)} & \qquad & \forall (t,x) \in (0,T) \times (0,1), \quad i=1,2,3, \label{PDEu} \\
u^{(i)}(t,1) &= P_{i} & \qquad & \forall t \in [0,T], \quad i=1,2,3, \nonumber \\
u^{(1)}(t,0) &= u^{(2)}(t,0)=u^{(3)}(t,0) & \qquad & \forall t \in [0,T], \nonumber \\ 
0 &= \sum_{i=1}^{3} \tau^{(i)}(t,0) + \epsilon u^{(i)}_{x}(t,0) & \qquad & \forall t\in [0,T], \label{JPC} \\
u^{(i)}(0,x) &= u^{(i)}_{0}(x) & \qquad & \forall x \in \domain, \quad i=1,2,3, \nonumber 
\end{align}

Observe that each curve moves according to a non-geometrical, i.e., parametrisation dependent perturbation of the so called special curvature flow \eqref{eq:csf2}. As shown above, integration by parts makes it possible to ``isolate'' the $\epsilon$-contribution to the normal component of the flow (see \eqref{PDEu}). Dealing with the weak form, as we do later on for the FEM-analysis, this ``decoupling'' seems no longer possible. Consequently, the parameter $\epsilon$ appears in all bounding constants of the error estimates for the numerical scheme, typically in an unfavourable way such that we can not provide estimates that hold true uniformly in $\epsilon$. 

We will be interested in approximating the solution on a finite time interval and make the following assumptions:

\begin{ass} \label{ass_cont_sol}
We assume the existence of a unique solution $\Gamma=(u^{(1)},u^{(2)},u^{(3)})$ to Problem~\ref{prob_cont} on some intervall $[0,T]$ such that, for each curve $i =1,2,3$, we have
\begin{align*}
&u^{(i)} \in L^{2}((0,T), W^{2,2}(\domain)),\\
&u^{(i)}_{t} \in L^{\infty}((0,T), W^{1,2}(\domain)) \cap L^{2}((0,T), W^{2,2}(\domain)),\\
&u^{(i)}_{0} \in W^{2,2}(\domain).
\end{align*}
Moreover, we assume that there is a small constant $c_{0} \in (0,\frac{1}{2}]$ such that for all $i=1,2,3$
\begin{align} \label{bc0}
 0 < c_{0} \leq |u_{x}^{(i)}(t,x)| \leq \frac{1}{c_{0}} \text{ on } [0,T] \times \domain.
\end{align}
\end{ass}

For any $b \in \R^{2}$ we have that $|b|^2 = (b \cdot \nu^{(i)})^2 + (b \cdot \tau^{(i)})^2$. Recalling that $\epsilon, c_{0} \leq \tfrac{1}{2}$, for any triod $\Upsilon = \{ v^{(1)},v^{(2)},v^{(3)} \} \in \triod_{P}$ we therefore obtain that 
\begin{align*}
 \langle v^{(i)}, v^{(i)} \rangle_{u^{(i)}} &= \int_{\domain} ( v^{(i)} \cdot \frac{(u_{x}^{(i)})^{\perp}}{|u_{x}^{(i)}|}) (v^{(i)} \cdot \frac{(u_{x}^{(i)})^{\perp}}{|u_{x}^{(i)}|}) |u_{x}^{(i)}| dx + \epsilon \int_{\domain} (v^{(i)} \cdot \frac{u_{x}^{(i)}}{|u_{x}^{(i)}|}) ( v^{(i)} \cdot \frac{u_{x}^{(i)}}{|u_{x}^{(i)}|})|u_{x}^{(i)}|^{2} dx \\
 &\geq c_{0} \int_{\domain} ( v^{(i)} \cdot \frac{(u_{x}^{(i)})^{\perp}}{|u_{x}^{(i)}|})^{2} dx + \epsilon c_{0}^{2} \int_{\domain} (v^{(i)} \cdot \frac{u_{x}^{(i)}}{|u_{x}^{(i)}|})^{2} dx \\
 &= c_{0} \underbrace{(1 - \epsilon c_{0})}_{\geq 3/4} \int_{\domain} ( v^{(i)} \cdot \frac{(u_{x}^{(i)})^{\perp}}{|u_{x}^{(i)}|})^{2} dx + \epsilon c_{0}^{2} \int_{\domain} | v^{(i)} |^2 dx.
\end{align*}
Moreover 
\[
 \langle v^{(i)}, v^{(i)} \rangle_{u^{(i)}} \leq \Big{(} \frac{1}{c_{0}} + \frac{\epsilon}{c_{0}^{2}} \Big{)} \int_{\domain} | v^{(i)} |^{2} dx
\]
and therefore for all $i=1,2,3$ and at all times $t \in [0,T]$ 
\begin{align} \label{eqnorms}
 \epsilon c_{0}^{2} \| v^{(i)} \|_{L^{2}(\domain)}^{2} \leq \langle v^{(i)}, v^{(i)} \rangle_{u^{(i)}(t)} \leq \frac{1}{c_{0}^{2}} \| v^{(i)} \|_{L^{2}(\domain)}^{2}.
\end{align}

%%%%%%%%%%%%%%%%%%%%%%%%%%%%%%%%%%%%%%%%%%%
\section{Finite elements and semi-discrete problem}

For the finite element approximation consider the uniform mesh with vertices $x_{j} = hj \in \domain$ for $j=0, \ldots, J$ with $h = 1/J$ for some $J \in \mbN$, and let $\domain_{j} = [x_{j-1},x_{j}]$, $j=1, \ldots, J$. We denote the space of continuous and piecewise linear functions on $\domain$ by 
\[
 S_{h} := \big{\{} v_{h} \in C^0(\domain,\R) \, \big{|} \, v_{h}|_{\domain_{j}} \mbox{ is linear} \big{\}}.
\]
The basis functions $\phi_{j} \in S_{h}$ are defined as usual through $\phi_{j}(x_{i})=\delta_{ij}$ for $i,j=0, \ldots, J$. 

Let $\interpol u$ denote the linear Lagrange interpolant. We shall use the standard interpolation estimates (both for scalar and vector valued functions) :
\begin{align}
\label{eq:interpolL2}
\| v - \interpol v \|_{L^{2}(\domain)} &\leq C_{p} h^k \|v \|_{W^{k,2}(\domain)} &\text{ for $k=1,2$}, \\
\label{eq:interpolH1}
\| (v - \interpol v)_x \|_{L^{2} (\domain)} & \leq C_{p} h \| v \|_{W^{2,2} (\domain)}, \\
\nonumber 
 \| (\interpol v)_{x} \|_{L^{2} (\domain)} & \leq C_{p} \| v_{x} \|_{ L^{2}(\domain)}, \\
\nonumber 
\| v - \interpol v \| _{L^{\infty}(\domain)} &\leq C_{p} h^{1/2} \| v_{x}\|_{L^{2}(\domain)}, \\
\nonumber 
\| (v - \interpol v)_{x} \| _{L^{\infty}(\domain)} &\leq C_{p} h^{1/2} \| v_{xx}\|_{L^{2}(\domain)}. 
\end{align}
Recall also the inverse estimates for any $w_{h} \in S_{h}$:
\begin{align}
\| w_{hx} \|_{L^{2}(\domain_{j})} & \leq \frac{C_{p}}{h} \| w_{h} \|_{L^{2} (\domain_{j})} & \quad \Longrightarrow \quad \|w_{hx} \|_{L^{2}(\domain)} & \leq \frac{C_{p}}{h} \| w_{h} \|_{L^{2} (\domain)}, \label{eq:inverse1} \\
\| w_{h} \|_{L^{\infty} (\domain_{j})} & \leq \frac{C_{p}}{\sqrt{h}} \| w_{h} \|_{L^{2}(\domain_{j}) } & \quad \Longrightarrow \quad \| w_{h} \|_{L^{\infty} (\domain)} & \leq \frac{C_{p}}{\sqrt{h}} \| w_{h} \|_{L^{2}(\domain)} \label{eq:inverse2}.
\end{align}

Similarly to the continuous setting we define discrete triods by
\begin{align*}
 \triod_{P,h} := \{ \, \Gamma_{h} = (u^{(1)}_{h},u^{(2)}_{h},u^{(3)}_{h}) \, | \, & u_{h}^{(i)} \in S^{2}_{h} \text{ regular almost everywhere}, \\
 & u_{h}^{(i)}(1) = P_{i}, \quad i=1,2,3, \\
 & u_{h}^{(1)}(0) = u_{h}^{(2)}(0) = u_{h}^{(3)}(0) \, \},
\end{align*}
and also introduce the notation 
\[
 \tau^{(i)}_{h} = \frac{u^{(i)}_{hx}}{|u^{(i)}_{hx}|} \quad \mbox{ and } \quad \nu^{(i)}_{h} = (\tau_{h}^{(i)})^\perp = \frac{(u^{(i)}_{hx})^{\perp}}{|u^{(i)}_{hx}|}.
\]
In case that $P_{1} = P_{2} = P_{3} = 0$ we write $\triod_{0,h}$ and note that this is a space of dimension 
\begin{equation} \label{eq:defdim}
 d_{0,h} := \dim(\triod_{0,h}) = 6J-4. 
\end{equation}
Note that \eqref{innerprod} is also well-defined for discrete triods, and even for functions $v,w \in W^{1,2}(\domain,\R^2)$ we can write 
\begin{align} \label{innerprodh}
 \langle v, w \rangle_{u^{(i)}_{h}} := \int_{\domain} (v \cdot \nu^{(i)}_{h}) (w \cdot \nu^{(i)}_{h}) |u^{(i)}_{hx}| + \epsilon (v \cdot \tau^{(i)}_{h}) (w \cdot \tau^{(i)}_{h}) |u^{(i)}_{hx}|^{2} dx, 
\end{align}
for $\Gamma_{h} = (u^{(1)}_{h},u^{(2)}_{h},u^{(3)}_{h}) \in \triod_{P,h}$ with uniformly bounded length elements. The semi-discrete problem that will be analysed for convergence reads: 

\begin{prob} \label{prob_semidis}
 Let $\Gamma_{0} = (u^{(1)}_{0},u^{(2)}_{0},u^{(3)}_{0}) \in \triod_{P}$ denote an initial triod with points $P_{i} \in \R^{2}$, $i=1,2,3$, such that Problem \ref{prob_cont} is well-posed on time interval $[0,T]$ as specified in Assumption~\ref{ass_cont_sol}. \\
 Find a family of discrete triods $\Gamma_{h}(t) = (u^{(1)}_{h}(t),u^{(2)}_{h}(t),u^{(3)}_{h}(t)) \in \triod_{P,h}$, $t \in [0,T]$, such that
 $u_{h}^{(i)}(0) = \interpol u^{(i)}_{0}$, $i=1,2,3$, and such that for all $t \in (0,T)$ and all $(\varphi^{(1)}_{h}, \varphi^{(2)}_{h}, \varphi^{(3)}_{h}) \in \triod_{0,h}$
 \begin{multline*}
 \sum_{i=1}^{3} \Bigg{(} \int_{\domain} (u_{ht}^{(i)} \cdot \frac{(u_{hx}^{(i)})^{\perp}}{|u_{hx}^{(i)}|}) (\varphi_{h}^{(i)} \cdot \frac{(u_{hx}^{(i)})^{\perp}}{|u_{hx}^{(i)}|}) |u_{hx}^{(i)}| dx + \epsilon \int_{\domain} (u_{ht}^{(i)} \cdot \frac{u_{hx}^{(i)}}{|u_{hx}^{(i)}|}) (\varphi_{h}^{(i)} \cdot \frac{u_{hx}^{(i)}}{|u_{hx}^{(i)}|}) |u_{hx}^{(i)}|^{2} dx \Bigg{)} \\
 = -\sum_{i=1}^{3} \Bigg{(} \epsilon \int_{\domain} u_{hx}^{(i)} \cdot \varphi_{hx}^{(i)} dx + \int_{\domain} \frac{u_{hx}^{(i)}}{|u_{hx}^{(i)}|} \cdot \varphi_{hx}^{(i)} dx \Bigg{)}.
 \end{multline*}
\end{prob}

%%%%%%%%%%%%%%%%%%%%%%%%%%%%%%%%%%%%%%%%%%%
\section{Convergence analysis}
\label{sec:convergence}

We now show that solutions to Problem \ref{prob_semidis} exist for $h$ small enough, and that they converge to the solution of Problem \ref{prob_cont}. The precise statement is below in Theorem \ref{th_conv}. It is proved using a fixed point argument. In the following, a generic constant $C$ may change from line to line. 

Let $\mathcal{Z}_{h} := C^{0}([0,T], S_{h}^{2})$ and $\mathcal{X}_{h} := \mathcal{Z}_{h}^{3}$ denote the Banach spaces of time continuous functions with values in $S_{h}^{2}$ and $(S_{h}^{2})^{3}$, respectively, endowed with the norms 
\[
 \| u_{h} \|_{\mathcal{Z}_{h}} := \sup_{ t \in [0,T]} \| u_{h}(t) \|_{L^{2}(\domain)}, \quad \|(u^{(1)}_{h}, u^{(2)}_{h}, u^{(3)}_{h}) \|_{\mathcal{X}_{h}}:= \max_{i=1,2,3} \| u^{(i)}_{h}\|_{\mathcal{Z}_{h}}.
\]
For some constants $K>1$, $M>0$ (to be specified later on) consider the set
\begin{align*}
\mathcal{B}_{h} := \big{\{} \, \Gamma_{h} = (u^{(1)}_{h}, u^{(2)}_{h}, u^{(3)}_{h}) \, \big{|} \, & u^{(i)}_{h} \in \mathcal{Z}_{h}, & \quad & i=1,2,3, \\
& \Gamma_{h}(t) \in \mathcal{T}_{P,h} & \quad & \forall t \in [0,T], \\
& u^{(i)}_{h}(0,\cdot) = (\interpol u^{(i)}_{0})(\cdot), & \quad & i=1,2,3, \\
& \sup_{t \in [0,T]} e^{-Mt} \| (u^{(i)}_{x} - u^{(i)}_{hx})(t) \|^{2}_{L^{2}(\domain)} \leq K^{2} h^{2}, & \quad & i=1,2,3 \, \big{\}}.
\end{align*}

In view of the application of the Schauder fixed point theory later on, let us briefly collect the relevant properties of the set $\mathcal{B}_{h} \subset \mathcal{X}_{h}$. 
\begin{enumerate}
 \item 
 $\mathcal{B}_{h}$ is non-empty if $K$ is big enough (which we assume henceforth): \\
 Consider the linear interpolation $(\interpol u^{(1)}, \interpol u^{(2)}, \interpol u^{(3)})$ of the given smooth solution $\Gamma$. Recalling Assumption~\ref{ass_cont_sol} we see that $\interpol u^{(i)} \in \mathcal{Z}_{h}$. Moreover, $\sup_{t \in [0,T]} \| u^{(i)}(t) \|_{W^{2,2}(\domain)}$ is finite as  $u^{(i)} \in W^{1,2}((0,T),W^{2,2}(\domain))$ by Assumption~\ref{ass_cont_sol}. With the interpolation inequality \eqref{eq:interpolH1} we then see that a constant that satisfies $K \geq C_{p} \sup_{t \in [0,T]} \| u^{(i)}(t) \|_{W^{2,2}(\domain)}$, $i=1,2,3$, is sufficient to ensure that 
 \[
 \sup_{t \in [0,T]} e^{-Mt} \| (u^{(i)}_{x} - (\interpol u^{(i)})_{x})(t) \|^{2}_{L^{2}(\domain)} \leq K^{2} h^{2}, \quad i=1,2,3. 
 \]

 \item 
 $\mathcal{B}_{h}$ is bounded: \\
 Using Assumption~\ref{ass_cont_sol} again, for any $\Gamma_{h}=(v_{h}^{(1)},v_{h}^{(2)},v_{h}^{(3)}) \in \mathcal{B}_{h}$ we have that 
 \begin{multline*}
 |v^{(i)}_{h}(t,x)| \leq |P_{i}| + \int_{0}^{1}|v^{(i)}_{hx}(t,x)| dx \\
 \leq |P_{i}| + \| v^{(i)}_{hx}(t, \cdot) - u^{(i)}_{x}(t) \|_{L^{2}(\domain)} + \|u^{(i)}_{x}(t)\|_{L^{2}(\domain)} \leq C + e^{MT} K^2 h^2. 
 \end{multline*}
 Taking the supremum over time and the $L^2$ norm in space we see that $\| \Gamma_{h} \|_{\mathcal{X}_{h}} \leq C$ with a constant $C>0$ independent of $\Gamma_{h} \in \mathcal{B}_{h}$. 
 
 \item 
 $\mathcal{B}_{h}$ is closed: \\
 Assume that $\|\Gamma_{h}^{(j)} - \Gamma_{h} \|_{\mathcal{X}_{h}} \to 0$ as $j \to \infty$ with $\Gamma_{h}^{(j)} = (u^{(1,j)}_{h},u^{(2,j)}_{h},u^{(3,j)}_{h}) \in \mathcal{B}_{h}$ and $\Gamma_{h} = (u^{(1)}_{h},u^{(2)}_{h},u^{(3)}_{h}) \in \mathcal{Z}_{h}^{3}$. 
 By the finite dimensionality of $S_{h}$ all norms are equivalent on that space and \eqref{eq:inverse1} holds, so also 
 \[
 \max_{i=1,2,3} \sup_{t \in [0,T]} e^{-Mt} \| u^{(i,j)}_{hx}(t) - u^{(i)}_{hx}(t) \|_{L^2(\domain)}^2 \leq \max_{i=1,2,3} \sup_{t \in [0,T]}
 \frac{C_{p}^{2}}{h^{2}} \| u^{(i,j)}_{h}(t) - u^{(i)}_{h}(t) \|_{L^2(\domain)}^2
 \to 0 
 \]
 as $j \to \infty$, whence the $h-$estimate is satisfied in the limit. Similarly, using \eqref{eq:inverse2} all pointwise conditions (boundaries, triple junction, and initial conditions) remain satisfied in the limit.
 
 \item 
 $\mathcal{B}_{h} \subset \mathcal{Z}_{h}^3$ is convex: \\
 Any convex combination clearly also satisfies the pointwise conditions, and the $h$-estimate is easy to show using the convexity of norms, too. 
\end{enumerate}

Given any $\Gamma_{h} = (u_{h}^{(1)}, u_{h}^{(2)}, u_{h}^{(3)}) \in \mathcal{B}_{h}$, using interpolation and inverse inequalities we can write
\begin{align*}
\| (u^{(i)}_{hx} &-u^{(i)}_{x})(t) \|_{L^{\infty}(\domain)} \leq \| (u^{(i)}_{x} -(\interpol u^{(i)})_{x})(t) \|_{L^{\infty}(\domain)} + \| (\interpol u^{(i)})_{x} -u^{(i)}_{hx})(t) \|_{L^{\infty}(\domain)} \\
&\leq C \sqrt{h} \| u^{(i)}_{xx}(t) \|_{L^{2}(\domain)} + \frac{C}{\sqrt{h}} \| (\interpol u^{(i)})_{x} - u^{(i)}_{hx})(t) \|_{L^{2}(\domain)} \\
&\leq C \sqrt{h} \| u^{(i)}_{xx}(t) \|_{L^{2}(\domain)} + \frac{C}{\sqrt{h}} ( \| (u^{(i)}_{x} - (\interpol u^{(i)})_{x})(t) \|_{L^{2}(\domain)} + \| (u^{(i)}_{x} - u^{(i)}_{hx})(t) \|_{L^{2}(\domain)} ) \\
&\leq C \sqrt{h} \| u^{(i)}(t) \|_{W^{2,2}(\domain)} + C \sqrt{h} K e^{\frac{MT}{2}},
\end{align*} 
where $C=C(C_{p})$. Hence, using Assumption~\ref{ass_cont_sol}, there is a (sufficiently small) $h_{0}=h_{0}(C_{p}, c_{0},K,M,T, \Gamma) > 0$ so that for all $h \leq h_{0}$ and $i=1,2,3$ 
\begin{align}\label{boundsLE}
 |u_{hx}^{(i)}(t,x)| \geq \frac{c_{0}}{2} \text{ and } |u_{hx}^{(i)}(t,x)| \leq \frac{2}{c_{0}} \quad \text{ for all } (t,x) \in [0,T] \times \domain.
\end{align}
Analogously to \eqref{eqnorms} one can now show that 
\begin{align} \label{eqnormsh}
 \epsilon \frac{c_{0}^{2}}{4} \| v^{(i)} \|_{L^{2}(\domain)}^{2} \leq \langle v^{(i)}, v^{(i)} \rangle_{u^{(i)}_{h}(t)} \leq \frac{4}{c_{0}^{2}} \| v^{(i)} \|_{L^{2}(\domain)}^{2}
\end{align}
for $i=1,2,3$ and $t \in [0,T]$.

Consider now the following problem:

\begin{prob} \label{prob_fixpoint}
 Given any $\Gamma_{h} = (u_{h}^{(1)}, u_{h}^{(2)}, u_{h}^{(3)}) \in \mathcal{B}_{h}$, find differentiable functions $Y_{h}^{(i)} \in \mathcal{Z}_{h}$ such that $(Y_{h}^{(1)}(t),Y_{h}^{(2)}(t),Y_{h}^{(3)}(t)) \in \triod_{P,h}$ for all $t \in [0,T]$, such that $Y_{h}^{(i)}(0) = \interpol u^{(i)}_{0}$, $i=1,2,3$, and such that for all $t \in (0,T)$ and all $(\varphi^{(1)}_{h}, \varphi^{(2)}_{h}, \varphi^{(3)}_{h}) \in \triod_{0,h}$
 \begin{multline}
 \sum_{i=1}^{3} \Bigg{(} \int_{\domain} (Y_{ht}^{(i)} \cdot \frac{(u_{hx}^{(i)})^{\perp}}{|u_{hx}^{(i)}|}) (\varphi_{h}^{(i)} \cdot \frac{(u_{hx}^{(i)})^{\perp}}{|u_{hx}^{(i)}|}) |u_{hx}^{(i)}| dx + \epsilon \int_{\domain} (Y_{ht}^{(i)} \cdot \frac{u_{hx}^{(i)}}{|u_{hx}^{(i)}|}) (\varphi_{h}^{(i)} \cdot \frac{u_{hx}^{(i)}}{|u_{hx}^{(i)}|}) |u_{hx}^{(i)}|^{2} dx \Bigg{)} \\
 = -\sum_{i=1}^{3} \Bigg{(} \epsilon \int_{\domain} Y_{hx}^{(i)} \cdot \varphi_{hx}^{(i)} dx + \int_{\domain} \frac{Y_{hx}^{(i)}}{|Y_{hx}^{(i)}|} \cdot \varphi_{hx}^{(i)} dx \Bigg{)}. \label{eq:Ynew}
 \end{multline}
\end{prob}

\begin{prop} \label{prop_fixpoint}
 Let $h \leq h_{0}=h_{0}(C_{p},c_{0},K,M,T, \Gamma, \epsilon)$.
 Problem \ref{prob_fixpoint} has a unique solution $(Y_{h}^{(1)},Y_{h}^{(2)},Y_{h}^{(3)})$ that depends continuously on $\Gamma_{h}$ and that satisfies the estimates
 \begin{align} 
 \sup_{t \in [0,T]} e^{-Mt} \| u_{x}^{(i)}(t) - Y_{hx}^{(i)}(t) \|_{L^{2}(\domain)}^{2} & \leq \Big{(} 1 + \frac{K^{2}}{M}\Big{)} C h^{2}, \label{eq:estim_prop} \\
 \int_{0}^{T} \| u_{t}^{(i)}(t') - Y_{ht}^{(i)}(t') \|_{L^{2}(\domain)}^{2} dt' & \leq \tilde{C} h^2, \label{eq:estim2_prop} 
 \end{align}
 for $i=1,2,3$, with a constant $C>0$ depending on $c_{0}$, $T$, $\epsilon$, $C_{p}$, and norms of the $u^{(i)}$ with respect to the spaces in Assumption~\ref{ass_cont_sol}, and a constant $\tilde{C}>0$ depending on the same parameters and $M$ and $K$. 
\end{prop}

\begin{proof}
Recalling \eqref{eq:defdim}, we may write 
\[
 (Y_{h}^{(1)}(t,x), Y_{h}^{(2)}(t,x),Y_{h}^{(3)}(t,x)) = \sum_{\alpha=1}^{d_{0,h}} y_{\alpha}(t) b_{\alpha}(x) + \sum_{i=1}^{3} P_{i} \hat{b}_{i}(x), 
\]
where $\{ b_{\alpha} \in S_{h}\, | \, \alpha = 1, \ldots, d_{0,h} \}$ are basis functions for $\triod_{0,h}$, the $y_{\alpha}(t) \in \R$ are coefficient functions, and the $\hat{b}_{i} \in S_{h}$ are such that $\triod_{P,h} = \triod_{0,h} + \sum_{i=1}^{3} P_{i} \hat{b}_{i}(x)$. Testing with the basis functions we transform the above system \eqref{eq:Ynew} into a system of ODEs the form 
\begin{align*}
 \boldsymbol{A}(t,\Gamma_{h}(t)) \dot{\boldsymbol{y}}(t) = \boldsymbol{f}(t,\Gamma_{h}(t),\boldsymbol{y}(t)).
\end{align*}
Here, $\boldsymbol{y}=(y_{1}, \ldots, y_{d_{0,h}})^{\top}$, the matrix $\boldsymbol{A}(t) \in \R^{d_{0,h} \times d_{0,h}}$ is positive definite on $[0,T]$ for $h \leq h_{0}$ thanks to \eqref{boundsLE} and \eqref{eqnormsh}, and $\boldsymbol{f}$ is locally Lipschitz in points $\boldsymbol{y}$ that are such that $Y^{(i)}_{hx} > 0$ for all $i$. As the latter is satisfied by the assumption on the initial data, short time existence and uniqueness thus follow by standard ODE theory. 

After eventually decreasing $h_{0}$, let $h \leq h_{0}$ be sufficiently small so that \eqref{boundsLE} and 
$$ |(\interpol u_{0}^{(i)})_{x}| \geq \frac{3c_{0}}{4}, \qquad |(\interpol u_{0}^{(i)})_{x}| \leq \frac{4}{3c_{0}} \qquad \text{ holds in } \domain \text{ for } i=1,2,3.$$
Without loss of generality let $0<T_{h}\leq T$ denote the maximal time for which
\begin{align}
 \label{boundsLEY}
|Y_{hx}^{(i)}| \geq \frac{c_{0}}{2} \text{ and } |Y_{hx}^{(i)}| \leq \frac{2}{c_{0}} \quad \text{holds in } [0,T_{h}] \times \domain \text{ for } i=1,2,3. 
\end{align}

From the weak formulations \eqref{eq:Ynew} and \eqref{eq:prob_weak} we infer that
\begin{align}
\langle u_{t}^{(i)}-Y_{ht}^{(i)}, \varphi_{h}^{(i)}\rangle_{u_{h}^{(i)}(t)} + \epsilon \int_{\domain} (u_{x}^{(i)} - Y_{hx}^{(i)}) \cdot \varphi_{hx}^{(i)} dx
 &+ \int_{\domain} ( \frac{u_{x}^{(i)}}{|u_{x}^{(i)}|} - \frac{Y_{hx}^{(i)}}{|Y_{hx}^{(i)}|} ) \cdot \varphi_{hx}^{(i)} dx \nonumber \\
 & = \langle u_{t}^{(i)}, \varphi_{h}^{(i)}\rangle_{u_{h}^{(i)}(t)} -\langle u_{t}^{(i)}, \varphi_{h}^{(i)}\rangle_{u^{(i)}(t)} \label{eq:weak_difference}
\end{align}
Let now $\varphi_{h}^{(i)} = \interpol (u_{t}^{(i)}) - Y_{ht}^{(i)}$, $i=1,2,3$. The interpolation ensures that $\interpol (u^{(i)})(t,1) = P_{i}$ for all $t$ so that $\interpol (u_{t}^{(i)})(t,1) = 0$, and also $Y_{ht}^{(i)}(t,1) = \pd_{t} (P_{i}) = 0$. Therefore $\varphi_{h}^{(i)} (t,1) = 0$ for all $t$ and $i$. Similarly, $\interpol (u^{(1)})(t,0) = \interpol (u^{(2)})(t,0) = \interpol (u^{(3)})(t,0)$, which also holds true for the $Y_{h}^{(i)}(t,0)$ by definition. Therefore $\varphi_{h}^{(1)}(t,0) = \varphi_{h}^{2}(t,0) = \varphi_{h}^{(3)}(t,0)$, and altogether $(\varphi_{h}^{(1)},\varphi_{h}^{(1)},\varphi_{h}^{(1)}) \in \triod_{0,h}$ is permitted as a test function in \eqref{eq:weak_difference}. Adding the left-hand-side of \eqref{eq:weak_difference} tested with $\varphi^{(i)} = u^{(i)}_{t}$ to both sides and putting the terms involving $\interpol (u_{t}^{(i)})$ to the right-hand-side we obtain that
\begin{align}
\langle u_{t}^{(i)}-Y_{ht}^{(i)}, & u_{t}^{(i)}-Y_{ht}^{(i)} \rangle_{u_{h}^{(i)}(t)} + \frac{d}{dt} \left( \frac{\epsilon}{2}\int_{\domain}| u_{x}^{(i)} - Y_{hx}^{(i)}|^{2} dx \right ) \nonumber \\
& \quad + \int_{\domain} (\frac{u^{(i)}_{x}}{|u^{(i)}_{x}|} -\frac{Y^{(i)}_{hx}}{|Y^{(i)}_{hx}|})\cdot (u^{(i)}_{t} -Y^{(i)}_{ht})_{x} dx \nonumber \displaybreak[0] \\
&= \langle u_{t}^{(i)}-Y_{ht}^{(i)}, u_{t}^{(i)} - \interpol u_{t}^{(i)} \rangle_{u_{h}^{(i)}(t)} \nonumber \\
& \quad + \epsilon \int_{\domain} (u_{x}^{(i)} - Y_{hx}^{(i)}) \cdot ( u_{t}^{(i)} - \interpol u_{t}^{(i)})_{x} dx \nonumber \\
& \quad + \left(\langle u_{t}^{(i)}, ( \interpol u_{t}^{(i)}-Y_{ht}^{(i)})\rangle_{u_{h}^{(i)}(t)} -\langle u_{t}^{(i)},(\interpol u_{t}^{(i)}-Y_{ht}^{(i)})\rangle_{u^{(i)}(t)} \right) \nonumber \\
& \quad + \int_{\domain} (\frac{u^{(i)}_{x}}{|u^{(i)}_{x}|} -\frac{Y^{(i)}_{hx}}{|Y^{(i)}_{hx}|}) \cdot (u^{(i)}_{t} -\interpol u^{(i)}_{t})_{x}\, dx \label{eq:start} \\
 & =: J_{1} + J_{2} + J_{3} + J_{4}. \nonumber 
\end{align}

Using \eqref{eqnormsh} shows that 
\begin{equation} \label{eq:estim_below}
 \langle u_{t}^{(i)}-Y_{ht}^{(i)}, u_{t}^{(i)}-Y_{ht}^{(i)} \rangle_{u_{h}^{(i)}(t)} \geq \epsilon \frac{c_{0}^{2}}{4} \| u_{t}^{(i)}-Y_{ht}^{(i)} \|_{L^{2}(\domain)}^{2}. 
\end{equation}
Another calculation shows that 
\begin{align*}
 \frac{d}{dt} \Bigg{(} & \frac{1}{2} \Big{|} \frac{u^{(i)}_{x}}{|u^{(i)}_{x}|} - \frac{Y^{(i)}_{hx}}{|Y^{(i)}_{hx}|} \Big{|}^{2} | Y^{(i)}_{hx}| \Bigg{)} \\
 &= \big{(} u^{(i)}_{xt} - Y^{(i)}_{hxt} \big{)} \cdot \Big{(} \frac{u^{(i)}_{x}}{|u^{(i)}_{x}|} - \frac{Y^{(i)}_{hx}}{|Y^{(i)}_{hx}|} \Big{)} \\
 & \quad - u^{(i)}_{xt} \cdot \Bigg{[} \Big{(} \frac{u^{(i)}_{x}}{|u^{(i)}_{x}|} - \frac{Y^{(i)}_{hx}}{|Y^{(i)}_{hx}|} \Big{)} \frac{|u^{(i)}_{x}| - |Y^{(i)}_{hx}|}{|u^{(i)}_{x}|} + \frac{u^{(i)}_{x}}{|u^{(i)}_{x}|} \, \frac{1}{2} \Big{|} \frac{u^{(i)}_{x}}{|u^{(i)}_{x}|} - \frac{Y^{(i)}_{hx}}{|Y^{(i)}_{hx}|} \Big{|}^{2} \frac{|Y^{(i)}_{hx}|}{|u^{(i)}_{x}|} \Bigg{]}. 
\end{align*}
Using this for the third term on the left-hand-side of \eqref{eq:start} we thus can write 
\begin{align} \label{eq:to_estimate}
\langle u_{t}^{(i)} - Y_{ht}^{(i)}, u_{t}^{(i)} - Y_{ht}^{(i)} \rangle_{u_{h}^{(i)}(t)} & + \frac{d}{dt} \left ( \frac{\epsilon}{2}\int_{\domain}| u_{x}^{(i)} - Y_{hx}^{(i)}|^{2} dx + \int_{\domain} \frac{1}{2} \Big{|} \frac{u^{(i)}_{x}}{|u^{(i)}_{x}|} -\frac{Y^{(i)}_{hx}}{|Y^{(i)}_{hx}|} \Big{|}^{2} | Y^{(i)}_{hx}| dx \right ) \nonumber \\
 & = J_{1} + J_{2} + J_{3} + J_{4} + J_{5},
\end{align}
where 
\[
 J_{5} = - \int_{\domain} u^{(i)}_{xt} \cdot \Bigg{[} \Big{(} \frac{u^{(i)}_{x}}{|u^{(i)}_{x}|} - \frac{Y^{(i)}_{hx}}{|Y^{(i)}_{hx}|} \Big{)} \frac{|u^{(i)}_{x}| - |Y^{(i)}_{hx}|}{|u^{(i)}_{x}|} + \frac{u^{(i)}_{x}}{|u^{(i)}_{x}|} \, \frac{1}{2} \Big{|} \frac{u^{(i)}_{x}}{|u^{(i)}_{x}|} - \frac{Y^{(i)}_{hx}}{|Y^{(i)}_{hx}|} \Big{|}^{2} \frac{|Y^{(i)}_{hx}|}{|u^{(i)}_{x}|} \Bigg{]} dx. 
\]

Let us now estimate the terms on the right-hand-side of \eqref{eq:to_estimate}. Using \eqref{eqnormsh} and interpolation estimate \eqref{eq:interpolL2} we infer that 
\begin{align*}
J_{1} &\leq \sqrt{ \langle u_{t}^{(i)}-Y_{ht}^{(i)}, u_{t}^{(i)}-Y_{ht}^{(i)} \rangle_{u_{h}^{(i)}(t)}} \sqrt{ \langle u_{t}^{(i)}-\interpol u_{t}^{(i)}, u_{t}^{(i)}-\interpol u_{t}^{(i)}\rangle_{u_{h}^{(i)}(t)}}
\\
& \leq \sqrt{ \langle u_{t}^{(i)}-Y_{ht}^{(i)}, u_{t}^{(i)}-Y_{ht}^{(i)} \rangle_{u_{h}^{(i)}(t)}} \, \, \frac{2}{c_{0}} \| u_{t}^{(i)}-\interpol u_{t}^{(i)} \|_{L^{2}(\domain)} \\
& \leq \tilde{\delta} \langle u_{t}^{(i)}-Y_{ht}^{(i)}, u_{t}^{(i)}-Y_{ht}^{(i)} \rangle_{u_{h}^{(i)}(t)} + \frac{4}{4 \tilde{\delta} c_{0}^{2}} \| u_{t}^{(i)}-\interpol u_{t}^{(i)} \|_{L^{2}(\domain)}^{2} \\
& \leq \tilde{\delta} \langle u_{t}^{(i)}-Y_{ht}^{(i)}, u_{t}^{(i)}-Y_{ht}^{(i)} \rangle_{u_{h}^{(i)}(t)} + \frac{C_{p}^{2}}{\tilde{\delta} c_{0}^{2}} h^{2} \|u_{t}^{(i)}\|^{2}_{W^{1,2}(\domain)}
\end{align*}
for some $\tilde{\delta} >0$ that will be chosen later on. 
Using \eqref{eq:interpolH1} we obtain that
\begin{align*}
J_{2} & \leq \epsilon \| u_{x}^{(i)} - Y_{hx}^{(i)} \|_{L^{2}(\domain)} \| (u_{t}^{(i)}-\interpol u_{t}^{(i)})_{x} \|_{L^{2}(\domain)} \\
& \leq \frac{\epsilon}{2} \| u_{x}^{(i)} - Y_{hx}^{(i)} \|_{L^{2}(\domain)}^{2} + \frac{\epsilon C_{p}^{2}}{2} h^{2} \|u_{t}^{(i)}\|_{W^{2,2}(\domain)}^{2}.
\end{align*}
Recalling \eqref{innerprod} and \eqref{innerprodh}, we can write 
\begin{align*}
J_{3} &= \int_{\domain} ( u_{t}^{(i)} \cdot (\nu_{h}^{(i)}-\nu^{(i)})( ( \interpol u_{t}^{(i)}-Y_{ht}^{(i)}) \cdot \nu_{h}^{(i)} ) |u_{hx}^{(i)}| dx \\ 
& \quad + \int_{\domain} ( u_{t}^{(i)} \cdot \nu^{(i)})( ( \interpol u_{t}^{(i)}-Y_{ht}^{(i)}) \cdot (\nu_{h}^{(i)} -\nu^{(i)}) ) |u_{hx}^{(i)}| dx \\
& \quad + \int_{\domain} ( u_{t}^{(i)} \cdot \nu^{(i)})( ( \interpol u_{t}^{(i)}-Y_{ht}^{(i)}) \cdot \nu^{(i)}) (|u_{hx}^{(i)}|-|u_{x}^{(i)}|) dx \\
& \quad + \epsilon \int_{\domain} (u_{t}^{(i)} \cdot (\tau_{h}^{(i)}-\tau^{(i)})( ( \interpol u_{t}^{(i)}-Y_{ht}^{(i)}) \cdot \tau_{h}^{(i)})|u_{hx}^{(i)}|^{2} dx \\
& \quad + \epsilon \int_{\domain} (u_{t}^{(i)} \cdot \tau^{(i)})( ( \interpol u_{t}^{(i)}-Y_{ht}^{(i)}) \cdot (\tau_{h}^{(i)}-\tau^{(i)} ))|u_{hx}^{(i)}|^{2} dx \\
& \quad + \epsilon \int_{\domain} (u_{t}^{(i)} \cdot \tau^{(i)})( ( \interpol u_{t}^{(i)}-Y_{ht}^{(i)}) \cdot \tau^{(i)})(|u_{hx}^{(i)}|^{2} - |u_{x}^{(i)}|^{2}) dx.
\end{align*}
Therefore, using $|u_{hx}^{(i)}|^{2} - |u_{x}^{(i)}|^{2} \leq (|u_{hx}^{(i)}| + |u_{x}^{(i)}|)(|u_{hx}^{(i)} - u_{x}^{(i)}|)$ in the last term and \eqref{bc0} and \eqref{boundsLE} we infer that 
\begin{align*}
J_{3} &\leq \frac{4}{c_{0}} \|u_{t}^{(i)}\|_{ L^{\infty}(\domain)} \| \nu^{(i)}-\nu_{h}^{(i)} \|_{L^{2}(\domain)} \| \interpol u_{t}^{(i)}-Y_{ht}^{(i)} \|_{L^{2}(\domain)} \\
& \quad + \|u_{t}^{(i)} \|_{ L^{\infty}(\domain)}\| u_{x}^{(i)}-u_{hx}^{(i)}\|_{L^{2}(\domain)} \| \interpol u_{t}^{(i)}-Y_{ht}^{(i)}\|_{L^{2}(\domain)} \\
& \quad + \epsilon \frac{8}{c_{0}^2} \|u_{t}^{(i)}\|_{ L^{\infty}(\domain)} \| \tau^{(i)}-\tau_{h}^{(i)} \|_{L^{2}(\domain)} \| \interpol u_{t}^{(i)}-Y_{ht}^{(i)} \|_{L^{2}(\domain)} \\
& \quad + \epsilon \frac{3}{c_{0}} \|u_{t}^{(i)} \|_{ L^{\infty}(\domain)} \| u_{x}^{(i)}-u_{hx}^{(i)}\|_{L^{2}(\domain)} \| \interpol u_{t}^{(i)}-Y_{ht}^{(i)}\|_{L^{2}(\domain)}.
\end{align*}
Again using \eqref{bc0}, a short calculation show that 
\[
 |\tau^{(i)}-\tau_{h}^{(i)}| \leq \frac{2}{c_{0}} |u^{(i)}_{x} - u^{(i)}_{hx}|, \quad |\nu^{(i)}-\nu_{h}^{(i)}| \leq \frac{2}{c_{0}} |u^{(i)}_{x} - u^{(i)}_{hx}|.
\]
Using furthermore that $\Gamma_{h} \in \mathcal{B}_{h}$, \eqref{eq:interpolL2}, $\epsilon \leq 1$, and the embedding $ W^{1,2}(\domain) \hookrightarrow L^{\infty}(\domain) $ we can deduce that 
\begin{align*}
J_{3} &\leq (\frac{8}{c_{0}^2} + 1) \|u_{t}^{(i)} \|_{ L^{\infty}(\domain)}\| u_{x}^{(i)}-u_{hx}^{(i)}\|_{L^{2}(\domain)} \| \interpol u_{t}^{(i)}-Y_{ht}^{(i)}\|_{L^{2}(\domain)} \\
& \quad + \epsilon (\frac{16}{c_{0}^3} + \frac{3}{c_{0}}) \|u_{t}^{(i)} \|_{ L^{\infty}(\domain)} \| u_{x}^{(i)}-u_{hx}^{(i)}\|_{L^{2}(\domain)} \| \interpol u_{t}^{(i)}-Y_{ht}^{(i)}\|_{L^{2}(\domain)} \displaybreak[0] \\
& \leq C(c_{0}) \|u_{t}^{(i)} \|_{ L^{\infty}(\domain)} K h e^{\frac{Mt}{2}} \big{(} C_{p} h \|u_{t}^{(i)}\|_{W^{1,2}(\domain)} + \| u_{t}^{(i)}-Y_{ht}^{(i)}\|_{L^{2}(\domain)} \big{)}\\
& \leq C h^{2} K e^{\frac{Mt}{2}} \|u_{t}^{(i)} \|_{W^{1,2}(\domain)}^{2} + \frac{C}{\epsilon \tilde{\delta}} K^{2} h^{2} e^{Mt} \|u_{t}^{(i)} \|_{ L^{\infty}(\domain)}^{2} + \tilde{\delta} \epsilon \| u_{t}^{(i)}-Y_{ht}^{(i)}\|_{L^{2}(\domain)}^{2} 
\end{align*}
with some $\tilde{\delta}>0$ to be chosen appropriately later on and a constant $C = C(c_{0},C_{p})$. 
Next, we have using \eqref{boundsLEY} and an interpolation inequality that
\begin{align*}
J_{4} \leq \frac{2}{c_{0}} \int_{\domain} \frac{1}{2} \Big{|} \frac{u^{(i)}_{x}}{|u^{(i)}_{x}|} - \frac{Y^{(i)}_{hx}}{|Y^{(i)}_{hx}|} \Big{|}^{2} |Y^{(i)}_{hx}| dx + C h^{2} \|u_{t}^{(i)}\|_{W^{2,2}(\domain)}^{2}
\end{align*}
with $C = C(C_{p})$. Finally, using \eqref{bc0} and \eqref{boundsLEY} we infer that 
\begin{align*}
J_{5} &\leq \frac{\|u_{xt}^{(i)}\|_{L^{\infty}(\domain)}}{c_{0}} \Bigg{[} \left \| \frac{u^{(i)}_{x}}{|u^{(i)}_{x}|} - \frac{Y^{(i)}_{hx}}{|Y^{(i)}_{hx}|} \right \|_{L^2(\domain)} \| |u^{(i)}_{x}| - |Y^{(i)}_{hx}| \|_{L^2(\domain)} + \int_{\domain} \frac{1}{2} \Big{|} \frac{u^{(i)}_{x}}{|u^{(i)}_{x}|} - \frac{Y^{(i)}_{hx}}{|Y^{(i)}_{hx}|} \Big{|}^{2} |Y^{(i)}_{hx}| dx \Bigg{]} \\
&\leq \|u_{xt}^{(i)}\|_{L^{\infty}(\domain)} \left(\frac{2}{c_{0}^{2}} + \frac{1}{c_{0}} \right) \int_{\domain} \frac{1}{2} \Big{|} \frac{u^{(i)}_{x}}{|u^{(i)}_{x}|} - \frac{Y^{(i)}_{hx}}{|Y^{(i)}_{hx}|} \Big{|}^{2} |Y^{(i)}_{hx}| dx + \frac{\|u_{xt}^{(i)}\|_{L^{\infty}(\domain)}}{2 c_{0}} \,\| u_{x}^{(i)} -Y_{hx}^{(i)}\|_{L^{2}(\domain)}^{2}\\
& \leq C(1 + \|u_{t}^{(i)}\|^{2}_{W^{2,2}(\domain)}) \int_{\domain} \frac{1}{2} \Big{|} \frac{u^{(i)}_{x}}{|u^{(i)}_{x}|} - \frac{Y^{(i)}_{hx}}{|Y^{(i)}_{hx}|} \Big{|}^{2} |Y^{(i)}_{hx}| dx 
+ C(1 + \|u_{t}^{(i)}\|^{2}_{W^{2,2}(\domain)})\| u_{x}^{(i)} -Y_{hx}^{(i)}\|_{L^{2}(\domain)}^{2}
\end{align*}
where $C=C(c_{0})$.
All in all, from \eqref{eq:to_estimate}, \eqref{eq:estim_below}, and the above estimates of the $J_{i}$ we obtain that 
\begin{align*}
 \frac{1}{2} \langle u_{t}^{(i)} & - Y_{ht}^{(i)}, u_{t}^{(i)} - Y_{ht}^{(i)} \rangle_{u_{h}^{(i)}} + \epsilon \frac{c_{0}^{2}}{8} \| u_{t}^{(i)}-Y_{ht}^{(i)} \|_{L^{2}(\domain)}^{2} \\ 
& \quad + \frac{d}{dt} \Bigg{(} \frac{\epsilon}{2} \| u_{x}^{(i)} - Y_{hx}^{(i)} \|_{L^{2}(\domain)}^{2} + \int_{\domain} \frac{1}{2} \Big{|} \frac{u^{(i)}_{x}}{|u^{(i)}_{x}|} - \frac{Y^{(i)}_{hx}}{|Y^{(i)}_{hx}|} \Big{|}^{2} |Y^{(i)}_{hx}| dx \Bigg{)} \\
 &\leq \tilde{\delta} \langle u_{t}^{(i)} - Y_{ht}^{(i)}, u_{t}^{(i)} - Y_{ht}^{(i)} \rangle_{u_{h}^{(i)}(t)} + \frac{1}{\tilde{\delta}} C h^{2} \\
 & \quad + \frac{\epsilon}{2} \| u_{x}^{(i)} - Y_{hx}^{(i)} \|_{L^{2}(\domain)}^{2} + C h^{2}\|u_{t}^{(i)}\|^{2}_{W^{2,2}(\domain)} \\
 & \quad + \tilde{\delta} \epsilon \| u_{t}^{(i)}-Y_{ht}^{(i)}\|_{L^{2}(\domain)}^{2} + C K e^{\frac{Mt}{2}} h^{2} + \frac{1}{\tilde{\delta}} C K^{2} e^{Mt} h^{2} \\
 & \quad + C (1 + \|u_{t}^{(i)}\|^{2}_{W^{2,2}(\domain)})\int_{\domain} \frac{1}{2} \Big{|} \frac{u^{(i)}_{x}}{|u^{(i)}_{x}|} - \frac{Y^{(i)}_{hx}}{|Y^{(i)}_{hx}|} \Big{|}^{2} |Y^{(i)}_{hx}| dx + C h^{2} \|u_{t}^{(i)}\|^{2}_{W^{2,2}(\domain)}\\
 & \quad + C (1 + \|u_{t}^{(i)}\|^{2}_{W^{2,2}(\domain)})\| u_{x}^{(i)} -Y_{hx}^{(i)}\|_{L^{2}(\domain)}^{2},
\end{align*}
where $C>0$ depends on $\epsilon$, $c_{0}$, $C_{p}$, and $\Gamma$ in terms of norms of the $u^{(i)}$ with respect to the spaces specified in Assumption~\ref{ass_cont_sol}. 
Note that $u^{(i)}_{t} \in L^{2}((0,T), W^{2,2}(\domain))$ only, whence we have to keep the term $\|u_{t}^{(i)}\|^{2}_{W^{2,2}(\domain)}$ until we later integrate with respect to time. 
Choosing now $\tilde{\delta} = c_{0}^{2} / 16 < 1/4$ (thanks to $c_{0} \leq 1$) we see that
\begin{align}
 \frac{1}{4} \langle u_{t}^{(i)} & - Y_{ht}^{(i)}, u_{t}^{(i)} - Y_{ht}^{(i)} \rangle_{u_{h}^{(i)}} + \epsilon \frac{c_{0}^{2}}{16} \| u_{t}^{(i)}-Y_{ht}^{(i)} \|_{L^{2}(\domain)}^{2} \nonumber \\ 
& \quad + \frac{d}{dt} \Bigg{(} \frac{\epsilon}{2} \| u_{x}^{(i)} - Y_{hx}^{(i)} \|_{L^{2}(\domain)}^{2} + \int_{\domain} \frac{1}{2} \Big{|} \frac{u^{(i)}_{x}}{|u^{(i)}_{x}|} - \frac{Y^{(i)}_{hx}}{|Y^{(i)}_{hx}|} \Big{|}^{2} |Y^{(i)}_{hx}| dx \Bigg{)} \nonumber \\
 & \leq C \big{(} 1 + K^{2} e^{Mt} \big{)} h^{2} + C h^{2}\|u_{t}^{(i)}\|^{2}_{W^{2,2}(\domain)} \label{eq:estim1} \\
 & \quad + C (1 + \|u_{t}^{(i)}\|^{2}_{W^{2,2}(\domain)})\Bigg{(} \frac{\epsilon}{2}\| u_{x}^{(i)} - Y_{hx}^{(i)} \|_{L^{2}(\domain)}^{2} + \int_{\domain} \frac{1}{2} \Big{|} \frac{u^{(i)}_{x}}{|u^{(i)}_{x}|} - \frac{Y^{(i)}_{hx}}{|Y^{(i)}_{hx}|} \Big{|}^{2} |Y^{(i)}_{hx}| dx \Bigg{)}. \nonumber
\end{align}
By Assumption~\ref{ass_cont_sol}, at time $t=0$ we have that 
\[
 \| u_{x}^{(i)}(0) - Y_{hx}^{(i)}(0) \|_{L^{2}(\domain)}^{2} = \| u_{0x}^{(i)} - (\interpol u_{0}^{(i)})_{x} \|_{L^{2}(\domain)}^{2} \leq C h^2 \| u_{0}^{(i)} \|_{W^{2,2}(\domain)}^{2}, \\
\]
and, using \eqref{boundsLE}, that 
\begin{align*}
 \int_{\domain} \frac{1}{2} \Big{|}\frac{u^{(i)}_{x}(0)}{|u^{(i)}_{x}(0)|} - \frac{Y^{(i)}_{hx}(0)}{|Y^{(i)}_{hx}(0)|} \Big{|}^{2} | Y^{(i)}_{hx}(0)| dx 
 & = \int_{\domain} \frac{1}{2} \Big{|}\frac{u^{(i)}_{0x}}{|u^{(i)}_{0x}|} - \frac{(\interpol u_{0}^{(i)})_{x}}{|(\interpol u_{0}^{(i)})_{x}|} \Big{|}^{2} |(\interpol u_{0}^{(i)})_{x}| dx \\
 & \leq C(c_{0}) \int_{\domain} |u^{(i)}_{0x} - (\interpol u_{0}^{(i)})_{x}|^{2} dx \leq C(c_{0}) h^2 \| u_{0}^{(i)} \|_{W^{2,2}(\domain)}^{2}. 
\end{align*}
Integrating \eqref{eq:estim1} on the time interval $(0,t)$ with $t \leq T_{h}$ we thus obtain that 
\begin{align}
\int_{0}^{t} & \frac{1}{4} \langle (u_{t}^{(i)}-Y_{ht}^{(i)})(t'), (u_{t}^{(i)}-Y_{ht}^{(i)})(t') \rangle_{u_{h}^{(i)}(t')} + \epsilon \frac{c_{0}^{2}}{16} \| u_{t}^{(i)}-Y_{ht}^{(i)} \|_{L^{2}(\domain)}^{2}(t') dt' \nonumber \\
& \quad + \frac{\epsilon}{2}\| u_{x}^{(i)}(t) - Y_{hx}^{(i)}(t) \|_{L^{2}(\domain)}^{2} + \int_{\domain} \frac{1}{2} \Big{|}\frac{u^{(i)}_{x}(t)}{|u^{(i)}_{x}(t)|} -\frac{Y^{(i)}_{hx}(t)}{|Y^{(i)}_{hx}(t)|} \Big{|}^{2} |Y^{(i)}_{hx}(t)| dx \nonumber \\
 & \leq C (1 + \frac{K^{2}}{M} e^{Mt}) h^{2} \nonumber \\
& \quad + C \int_{0}^{t} (1 + \|u_{t}^{(i)} (t')\|^{2}_{W^{2,2}(\domain)})\Bigg{(} \frac{\epsilon}{2}\| u_{x}^{(i)}(t') - Y_{hx}^{(i)}(t') \|_{L^{2}(\domain)}^{2} \nonumber \\
& \qquad \qquad + \int_{\domain} \frac{1}{2} \Big{|} \frac{u^{(i)}_{x} (t')}{|u^{(i)}_{x}(t')|} -\frac{Y^{(i)}_{hx} (t')}{|Y^{(i)}_{hx}(t')|} \Big{|}^{2} |Y^{(i)}_{hx}(t')| dx \Bigg{)} dt' \label{eq:estim2}
\end{align}
where $C>0$ depends on $\epsilon$, $c_{0}$, $C_{p}$, $T$, and $\Gamma$. 
A Gronwall argument now yields that
\begin{multline} \label{eq:stima}
 \frac{\epsilon}{2} \| u_{x}^{(i)}(t) - Y_{hx}^{(i)}(t) \|_{L^{2}(\domain)}^{2} 
 + \int_{\domain} \frac{1}{2} \Big{|} \frac{u^{(i)}_{x}(t)}{|u^{(i)}_{x}(t)|} - \frac{Y^{(i)}_{hx}(t)}{|Y^{(i)}_{hx}(t)|} \Big{|}^{2} |Y^{(i)}_{hx}(t)| dx \\
\leq C (1 + \frac{K^{2}}{M} e^{Mt}) h^{2}.
\end{multline}

Using the same ideas employed to show \eqref{boundsLE}, we can choose $h_{0} = h_{0}(C_{p},T,K,M, \epsilon, \Gamma)$ even smaller to ensure that \eqref{boundsLEY} is satisfied with strict inequality signs. This gives a contradiction to the maximality of $T_{h}$. Hence $T_{h} = T$ as claimed. 

Moreover, all estimates obtained so far hold on the whole time interval $[0,T]$. We can deduce \eqref{eq:estim_prop} from \eqref{eq:stima}. The other estimate \eqref{eq:estim2_prop} is then obtained from incorporating \eqref{eq:stima} into \eqref{eq:estim2} and absorbing all constants into $\tilde{C}$. 

Continuous dependence of the solution $(Y_{h}^{(1)}(t),Y_{h}^{(2)}(t),Y_{h}^{(3)}(t))$ on the data (in particular, on $\Gamma_{h}$) follows from standard ODE theory. For instance, see \cite{Teschl2012ODEs}, Theorem 2.8, where we note that convergence $\| u_{h,j} - u_{h} \|_{\mathcal{Z}_{h}} \to 0$ as $j \to \infty$ for functions $u_{h,j}, u_{h} \in \mathcal{Z}_{h}$ also implies that 
\[
 \sup_{t \in [0,T]} \| u_{h,j} - u_{h} \|_{W^{1,\infty}(\domain)} \to 0 \quad \mbox{as } j \to \infty
\]
because $S_{h}$ is finite dimensional and, thus, norms are equivalent on this space. 
\end{proof}

This result including the stability estimates is key for the fixed point argument that we use to establish the following convergence result:

\begin{teo} \label{th_conv}
Let $h \leq h_{0}=h_{0}(\epsilon, T, \Gamma, c_{0}, C_{p})$. Problem~\ref{prob_semidis} admits a unique solution $\Gamma_{h}$ with $\Gamma_{h}(t) = (u^{(1)}_{h}(t), u^{(2)}_{h}(t), u^{(3)}_{h}(t)) \in \triod_{P,h}$, $t \in [0,T]$ that satisfies the estimates 
\begin{align}\label{error-est}
\int_{0}^{T}\| u_{t}^{(i)}-u_{ht}^{(i)} \|_{L^{2}(\domain)}^{2}(t') dt' + \max_{t \in [0,T]} \| u_{x}^{(i)}(t) - u_{hx}^{(i)}(t) \|_{L^{2}(\domain)}^{2} \leq C h^{2},
\end{align}
for $i=1,2,3$, and a constant $C>0$ depending on $c_{0}$, $T$, $\epsilon$, $C_{p}$, and norms of the $u^{(i)}$ as in Assumption~\ref{ass_cont_sol}. 
\end{teo}

\begin{proof}
On the non-empty, convex, bounded, closed set $\mathcal{B}_{h}\subset \mathcal{X}_{h}$ consider the operator
\begin{align*} 
F: \, &\mathcal{B}_{h} \to C^{0}([0,T_{h}], S^{2}_{h})^{3}, \quad 
\Gamma_{h} \mapsto F(\Gamma_{h}) := (Y_{h}^{(1)}, Y_{h}^{(2)}, Y_{h}^{(3)})
\end{align*}
where the maps $Y_{h}^{(i)} \in \mathcal{Z}_{h}$, $i=1,2,3$, are the solution to Problem \ref{prob_fixpoint} from Proposition \ref{prop_fixpoint}. By that proposition $F$ is a continuous map. 

We choose $K$ and $M$ such that $K^{2} \geq 2C$ and $M \geq 2C$, with $C$ the constant appearing in \eqref{eq:estim_prop}. Then $(1 + \tfrac{K^{2}}{M} ) C \leq K^2$, and from \eqref{eq:estim_prop} we obtain that 
\begin{align*}
 \sup_{ t \in [0,T]}e^{-Mt}\| u_{x}^{(i)}(t) - Y_{hx}^{(i)}(t) \|_{L^{2}(\domain)}^{2} \leq K^{2} h^{2} \quad \forall t \in [0,T]. 
\end{align*}
This implies that $F(\mathcal{B}_{h}) \subset \mathcal{B}_{h}$. 

By \eqref{eq:estim2_prop} and the fact that $Y_{h}^{(i)}(0)=I_{h}u_{0}^{(i)}$, $i=1,2,3$, it follows that $F(\mathcal{B}_{h})$ is a bounded subset of $W^{1,2}((0,T),S_{h}^2)^3$. As $S_{h}$ is finite dimensional, the embedding $W^{1,2}((0,T),S_{h}^2)^3 \hookrightarrow C^{0}([0,T],S_{h}^2)^3$ is compact. Therefore, $F$ is a compact operator. 

The Schauder fixed point theorem thus yields the existence of a fixed point $F(\Gamma_{h})=\Gamma_{h}$. The error estimate \eqref{error-est} for this fixed point follows immediately from \eqref{eq:estim_prop} and \eqref{eq:estim2_prop}. 

Regarding uniqueness one can proceed as in the proof of Proposition \ref{prop_fixpoint} by formulating the problem as an ODE. The properties of the initial data ensure short time uniqueness, and thanks to the error estimates this argument can be extended to the whole time interval. 
\end{proof}

%%%%%%%%%%%%%%%%%%%%%%%%%%%%%%%%%%%%%%%%%%%
\section{Numerical tests}

\subsection{Time discretisation}

To validate the theoretical findings and further explore the properties of the finite element scheme we discretise in time with a simple first order IMEX-scheme so that a linear problem is obtained in each time step. 

Let $\delta = T/N > 0$ denote the time step size for some $N \in \mbN$. Let $t_{n} := n \delta$, $n=0, \dots, N$, and we write $u^{(i),n}$ for the approximation of $u^{(i)}(t_{n}, \cdot )$. 

\begin{prob} \label{prob_volldis}
Let $\Gamma_{h}^{0}=(U^{(1),0}, U^{(2),0}, U^{(3),0})= (\interpol u_{0}^{(1)}, \interpol u_{0}^{(2)}, \interpol u_{0}^{(3)})$. For $n =1,2 , \ldots,N$ compute $\Gamma_{h}^{n} = (U^{(1),n},U^{(2),n},U^{(3),n}) \in \triod_{P,h}$, such that for all $(\varphi^{(1)}_{h}, \varphi^{(2)}_{h}, \varphi^{(3)}_{h}) \in \triod_{0,h}$
 \begin{multline}\label{weakF}
 \sum_{i=1}^{3} \Bigg{(} \int_{\domain} (\frac{U^{(i),n} - U^{(i),n-1}}{\delta} \cdot \frac{(U_{x}^{(i),n-1})^{\perp}}{|U_{x}^{(i),n-1}|}) (\varphi_{h}^{(i)} \cdot \frac{(U_{x}^{(i), n-1})^{\perp}}{|U_{x}^{(i),n-1}|}) |U_{x}^{(i), n-1}| dx \\
 + \epsilon \int_{\domain} (\frac{U^{(i),n} - U^{(i),n-1}}{\delta} \cdot \frac{U_{x}^{(i), n-1}}{|U_{x}^{(i), n-1}|}) (\varphi_{h}^{(i)} \cdot \frac{U_{x}^{(i), n-1}}{|U_{x}^{(i), n-1}|}) |U_{x}^{(i), n-1}|^{2} dx \Bigg{)} \\
 +\sum_{i=1}^{3} \Bigg{(} \epsilon \int_{\domain} U_{x}^{(i), n} \cdot \varphi_{hx}^{(i)} dx + \int_{\domain} \frac{U_{x}^{(i),n}}{|U_{x}^{(i), n-1}|} \cdot \varphi_{hx}^{(i)} dx \Bigg{)}=0.
 \end{multline}
\end{prob}

Problem \eqref{weakF} can be written as a system of linear equations that incorporates the boundary and triple junction conditions. Let $e_{1} = (1,0),\, e_{2} = (0,1) \in \R^2$ and recall the notation $\phi_{j}$ for the standard basis functions of $S_{h}$. For $i=1,2,3$ and $m=n-1,n$ let us write 
\[
 U^{(i),m} = \sum_{k=0,\beta=1}^{J,2} U^{(i),m}_{k,\beta} e_{\beta} \phi_{k}, \quad \underline{U}^{(i),m} = \big{(} (U^{(i),m}_{k,1})_{k=0}^{J}, (U^{(i),m}_{k,2})_{k=0}^{J} \big{)} \in \R^{2(J+1)}. 
\]
Define now the symmetric tridiagonal matrices $M^{(i),n-1}, S^{(i),n-1} \in \R^{2(J+1) \times 2(J+1)}$ with the entries 
\begin{align}
 M^{(i),n-1}_{j,k,\alpha,\beta} &:= \int_{\domain} \frac{1}{\delta} (e_{\beta} \phi_{k} \cdot \frac{(U_{x}^{(i),n-1})^{\perp}}{|U_{x}^{(i),n-1}|}) (e_{\alpha} \phi_{j} \cdot \frac{(U_{x}^{(i), n-1})^{\perp}}{|U_{x}^{(i),n-1}|}) |U_{x}^{(i), n-1}| dx \nonumber \\
 & \quad + \epsilon \int_{\domain} \frac{1}{\delta} (e_{\beta} \phi_{k} \cdot \frac{U_{x}^{(i), n-1}}{|U_{x}^{(i), n-1}|}) (e_{\alpha} \phi_{j} \cdot \frac{U_{x}^{(i), n-1}}{|U_{x}^{(i), n-1}|}) |U_{x}^{(i), n-1}|^{2} dx, \label{eq:matmass} \\
 S^{(i),n-1}_{j,k,\alpha,\beta} &:= \int_{\domain} \epsilon e_{\beta} \partial_{x} \phi_{k} \cdot e_{\alpha} \partial_{x} \phi_{j} + \frac{1}{|U_{x}^{(i), n-1}|} e_{\beta} \partial_{x} \phi_{k} \cdot e_{\alpha} \partial_{x} \phi_{j} dx, \label{eq:matstiff}
\end{align}
for $j,k = 0, \dots, J$ and $\alpha,\beta = 1,2$. To incorporate the Dirichlet boundary conditions $U^{(i),n+1}(1) = P_{i}$ the rows corresponding to $j=J$ in \eqref{eq:matmass} and \eqref{eq:matstiff} and the right-hand-side of the system of linear equations are amended as usual. With regards to the other end of the curves consider the space
\[
 \tilde{\mathcal{T}}_{h} := \{ (w^{(1)}_{h},w^{(2)}_{h},w^{(3)}_{h}) \in (S_{h}^{2})^{3} \, | \, w^{(1)}_{h}(0) = w^{(2)}_{h}(0) = w^{(3)}_{h}(0) \}
\]
and the projection $\mathcal{P}_{h} : (S_{h}^2)^3 \to \tilde{\mathcal{T}}_{h}$ defined as follows: Choosing again the $e_{\beta} \phi_{k}$ as a basis of $S_{h}^{2}$, its corresponding matrix is denoted by $P \in \R^{6(J+1) \times 6(J+1)}$ and defines the linear map 
\[
 P : \big{(} (V^{(i)}_{k,1})_{k=0}^{J}, (V^{(i)}_{k,2})_{k=0}^{J} \big{)}_{i=1}^3 \mapsto \big{(} (W^{(i)}_{k,1})_{k=0}^{J}, (W^{(i)}_{k,2})_{k=0}^{J} \big{)}_{i=1}^3 
\]
where for $i=1,2,3$ and $\beta = 1,2$ 
\[ 
 W^{(i)}_{0,\beta} = \frac{1}{3} \sum_{\ell=1}^{3} V^{(\ell)}_{0,\beta}, \qquad W^{(i)}_{k,\beta} = V^{(i)}_{k,\beta}, \quad k=1, \dots, J. 
\] 
Note that then the functions $w^{(i)}_{h} \in S_{h}^{2}$, $i=1,2,3$, given by $w^{(i)}_{h} = \sum_{k=0,\beta=1}^{J,2} W^{(i),m}_{k,\beta} e_{\beta} \phi_{k}$ indeed satisfy $w^{(1),n}_{h}(0) = w^{(2),n}_{h}(0) = w^{(3),n}_{h}(0)$. Moreover, the matrix $P$ is symmetric. We also remark that the functions $(\mathcal{P}_{h}(e_{\alpha} \phi_{j}),0,0)$, $(0,\mathcal{P}_{h}(e_{\alpha} \phi_{j}),0)$, and $(0,0,\mathcal{P}_{h}(e_{\alpha} \phi_{j}))$ for $\alpha=1,2$ and $j=0, \dots, J-1$ span $\triod_{0,h}$. Furthermore, we note that $\mathcal{P}_{h} (U^{(1),m},U^{(2),m},U^{(3),m}) = (U^{(1),m},U^{(2),m},U^{(3),m})$, $m=n-1,n$, as both triples are elements of $\triod_{P,h}$. Altogether, the problem \eqref{weakF} can thus be written in the matrix-vector form 
\begin{multline} \label{eq:SLE}
 P 
 \begin{pmatrix} 
 (M+S)^{(1),n-1} & 0 & 0 \\ 
 0 & (M+S)^{(2),n-1} & 0 \\
 0 & 0 & (M+S)^{(3),n-1} 
 \end{pmatrix} 
 P 
 \begin{pmatrix}
 \underline{U}^{(1),n} - \underline{U}^{(1),n-1} \\
 \underline{U}^{(2),n} - \underline{U}^{(2),n-1} \\
 \underline{U}^{(3),n} - \underline{U}^{(3),n-1} 
 \end{pmatrix} 
 \\
 = - P 
 \begin{pmatrix} 
 S^{(1),n-1} & 0 & 0 \\ 
 0 & S^{(2),n-1} & 0 \\
 0 & 0 & S^{(3),n-1} 
 \end{pmatrix} 
 P 
 \begin{pmatrix}
 \underline{U}^{(1),n-1} \\
 \underline{U}^{(2),n-1} \\
 \underline{U}^{(3),n-1} 
 \end{pmatrix} 
\end{multline}
Whilst the projection matrix is symmetric, the system matrix is not after manipulating the matrix entries of $M^{(i),n-1}$ and $S^{(i),n-1}$ to incorporate the Dirichlet boundary conditions. Nevertheless, we could use a conjugate gradient iteration to solve the system thanks to choosing the previous solution as initial guess. The corresponding finite element functions satisfy the Dirichlet boundary conditions already, whence the residuals and, thus, all search directions and iterates are in a subspace restricted to which the system matrix is a symmetric operator. 

\begin{table}
{\footnotesize
\begin{center}
\begin{tabular}{|r|l||l|l|r|l|} \hline 
$l$ & $\epsilon_{l} = 0.3^{l-1}$ & $\lambda_{\max}(\epsilon_{l})$ & $\lambda_{\min}(\epsilon_{l})$ & $\cond_{2}(\epsilon_{l})$ & $\eoc_{l-1,l}$ \\ \hline \hline 
 1 & 1 & 2.0025 & 0.33758 & 5.9 & -- \\ \hline 
 2 & 0.3 & 2.5482 & 0.14957 & 17.0 & -0.8763 \\ \hline 
 3 & 0.09 & 2.8415 & 0.050742 & 56.0 & -0.9884 \\ \hline 
 4 & 0.027 & 2.9451 & 0.016172 & 182.1 & -0.9795 \\ \hline 
 5 & 0.0081 & 2.9787 & 0.0051151 & 582.3 & -0.9655 \\ \hline 
 6 & 0.00243 & 2.9894 & 0.0016401 & 1822.7 & -0.9478 \\ \hline 
 7 & 0.000729 & 2.9928 & 0.00054014 & 5540.8 & -0.9234 \\ \hline 
 8 & 0.0002187 & 2.9939 & 0.00018427 & 16247.0 & -0.8935 \\ \hline 
 9 & 6.561e-05 & 2.9952 & 6.4619e-05 & 46351.0 & -0.8707 \\ \hline 
10 & 1.9683e-05 & 2.9964 & 2.1764e-05 & 137680.0 & -0.9042 \\ \hline 
11 & 5.9049e-06 & 2.9968 & 6.8319e-06 & 438640.0 & -0.9624 \\ \hline 
\end{tabular}
\end{center}
}
\caption{Data on the diagonal block matrix $\diag(M^{(1),0},M^{(2),0},M^{(3),0})$ with $M^{(i),0}$ as defined in 
(\ref{eq:matmass}) after row equilibration (division of each row by the diagonal entry). 
For several decreasing values of $\epsilon$ we list the largest eigenvalue $\lambda_{\max}(\epsilon)$, 
the smallest eigenvalue $\lambda_{\min}(\epsilon)$, the condition number 
$\cond_{2} = \lambda_{\max}(\epsilon) / \lambda_{\min}(\epsilon)$ and its \emph{experimental order of convergence} 
$\eoc_{l,l-1} = (\log(cond_{2}(\epsilon_{l-1})) - \log(cond_{2}(\epsilon_{l}))) / (\log(\epsilon_{l-1}) - \log(\epsilon_{l})$. 
The functions $U^{(i),0}$ required for the assembly were from the example defined in Subsection \ref{subsec:conv}, and 
the discretisation parameters $J=20$, $h = 0.05$ and $\delta = h^2 = 0.0025$ were fixed.
} 
\label{tab:conditioning}
\end{table}

\begin{table}
{\footnotesize
\begin{center}
\begin{tabular}{|r|l|l||r|r|r|} \hline 
$J$ & $h$ & $\delta$ & $\cond_{2}(10^{-1})$ & $\cond_{2}(10^{-5})$ & ratio \\ \hline \hline 
 10 & 0.1 & 0.004 & 55.34 & 113.9 & 2.058 \\ \hline 
 16 & 0.0625 & 0.0015625 & 98.29 & 281.4 & 2.863 \\ \hline 
 24 & 0.041667 & 0.00069444 & 142.37 & 622.32 & 4.371 \\ \hline 
 36 & 0.027778 & 0.00030864 & 181.33 & 1380.7 & 7.614 \\ \hline 
 48 & 0.020833 & 0.00017361 & 202.07 & 2425.5 & 12.003 \\ \hline 
 64 & 0.015625 & 9.7656e-05 & 217.28 & 4239.1 & 19.510 \\ \hline 
\end{tabular}
\end{center}
}
\caption{Condition numbers (ratio of the largest to the smallest eigenvalue) for $\epsilon = 10^{-1}$ (column 4) and 
$\epsilon = 10^{-5}$ (column 5) of the system matrix in \eqref{eq:SLE} ($n=1$) for varying values of $J$ with $h=1/J$ and 
$\delta = 0.4h^2$. The last column contains the ratio $\cond_{2}(10^{-5}) / \cond_{2}(10^{-1})$. 
The functions $U^{(i),0}$ required for the assembly were chosen as in the example defined in Subsection \ref{subsec:conv}.} 
\label{tab:conditioning2}
\end{table}

\begin{rem}[Impact of $\epsilon$ on the conditioning] \label{rem:conditioning}
 As the convergence speed of the conjugate gradient method typically depends on the conditioning we looked at the impact of $\epsilon$ on the ratio of the largest to the smallest eigenvalue of the matrices in \eqref{eq:SLE}. \\
 Regarding the mass matrices $M^{(i),n-1}$ defined in \eqref{eq:matmass}, the tangential contributions scale linearly in $\epsilon$ in contrast to the normal contributions. Consequently, for fixed step sizes in space and time, the smallest eigenvalue scales with $\epsilon$ and the largest remains of order one. This is also what we observe in practice, see Table \ref{tab:conditioning} for typical data. The $\eocs$ for the condition numbers are close to $-1$ indicating a scaling with $\epsilon^{-1}$. \\
 In turn, the stiffness matrices $S^{(i),n-1}$ defined in \eqref{eq:matstiff} do not degenerate as $\epsilon \to 0$ as long as the length element $| U^{(i),n-1}_{x} |$ doesn't change significantly in dependence of $\epsilon$. In practice, the condition numbers of these matrices display the usual scaling with $h^{-2}$ rather independently of $\epsilon$ and therefore are not explicitely listed. \\
 The projection matrix $P$ does not depend on $\epsilon$ and barely has any impact on the overall conditioning. In conclusion, for relatively large time steps the conditioning is dominated by the stiffness contribution and, thus, by the spatial step size. But if the time steps are relatively small then the $\epsilon$-dependent conditioning of the mass matrix can become dominant. For a typical choice of $\delta = 0.4 h^2$ used in our simulations later on, values of the condition numbers of the full system matrix for differing values of $\epsilon$ are displayed in Table \ref{tab:conditioning2}. The increased conditioning for small $\epsilon$ was felt in terms of higher CG iteration numbers to obtain a given tolerance. But the overall computation times were still acceptable in our simulations, whence no preconditioning was considered. 
\end{rem}

\subsection{Validation of the convergence result}
\label{subsec:conv}

\begin{figure}
\begin{center}
 \includegraphics[width=5cm]{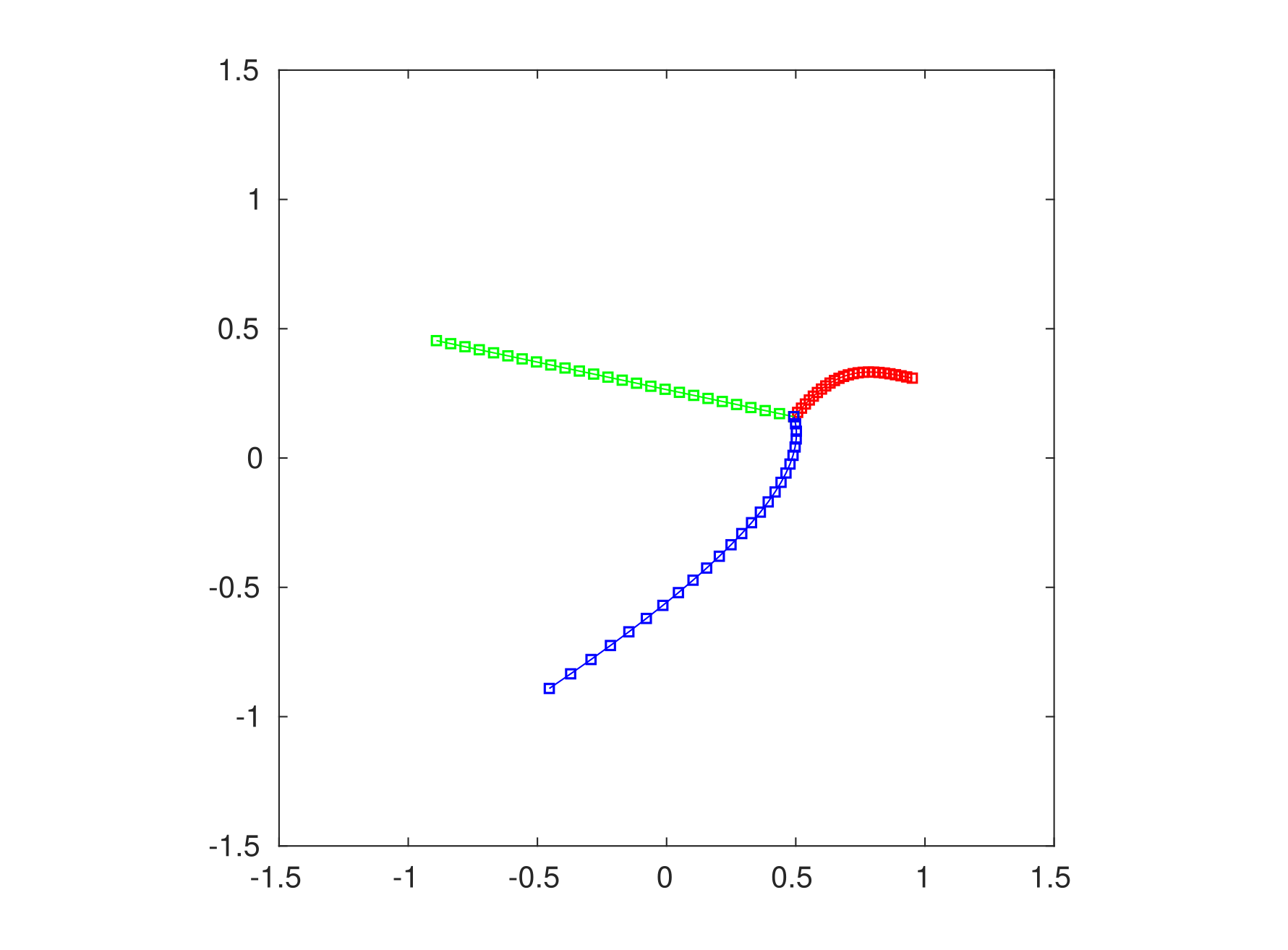} 
 \hfill 
 \includegraphics[width=5cm]{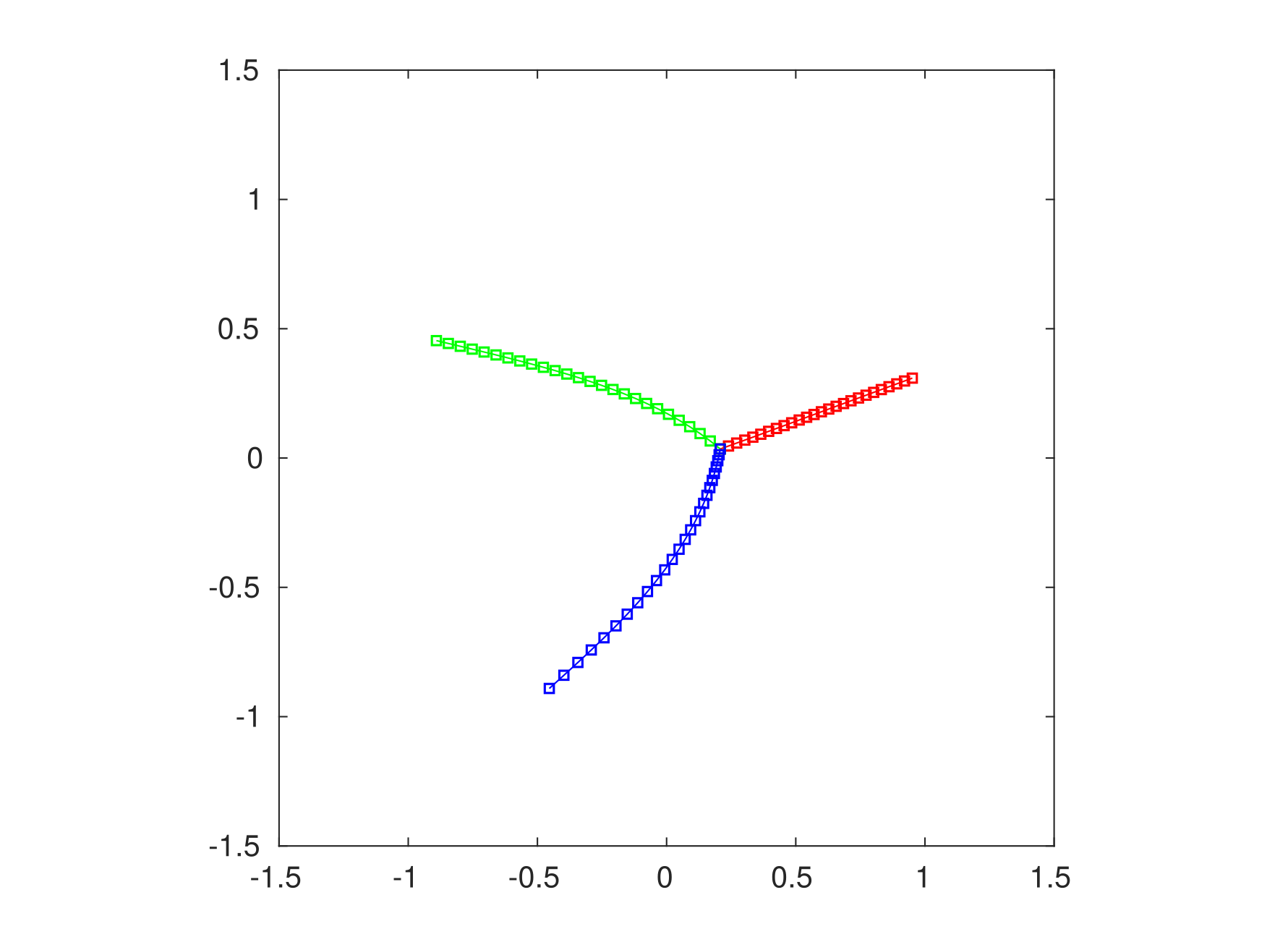} 
 \hfill 
 \includegraphics[width=5cm]{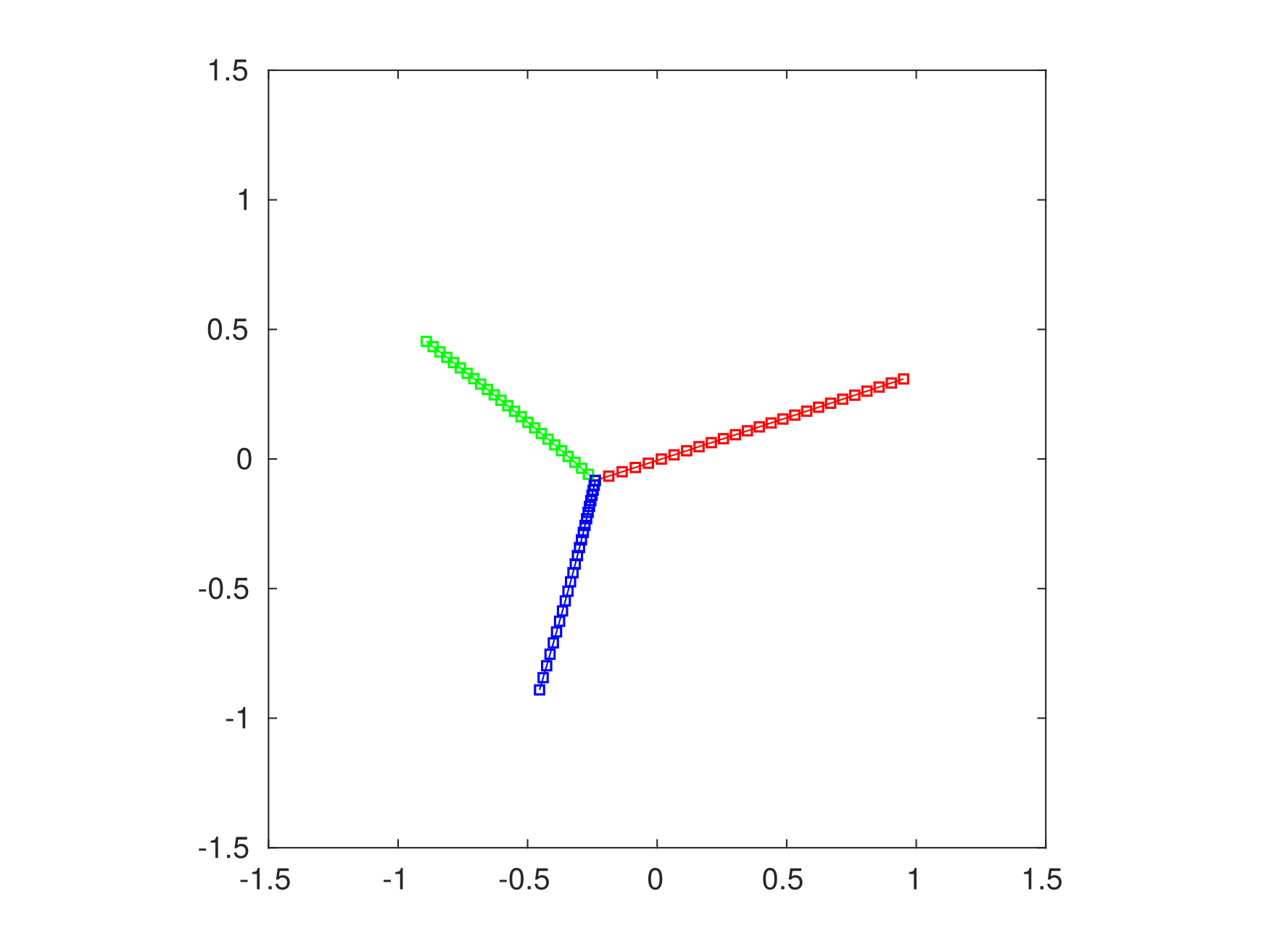} 
\end{center}
\caption{For the convergence test in Subsection \ref{subsec:conv}: Initial configuration (left), configuration at the final time $T=0.2$ for the error computations (middle), and result of a longer simulations at time $T=1.0$.}
\label{fig:conv_config}
\end{figure}

\begin{figure}
\begin{center}
 \includegraphics[width=10cm]{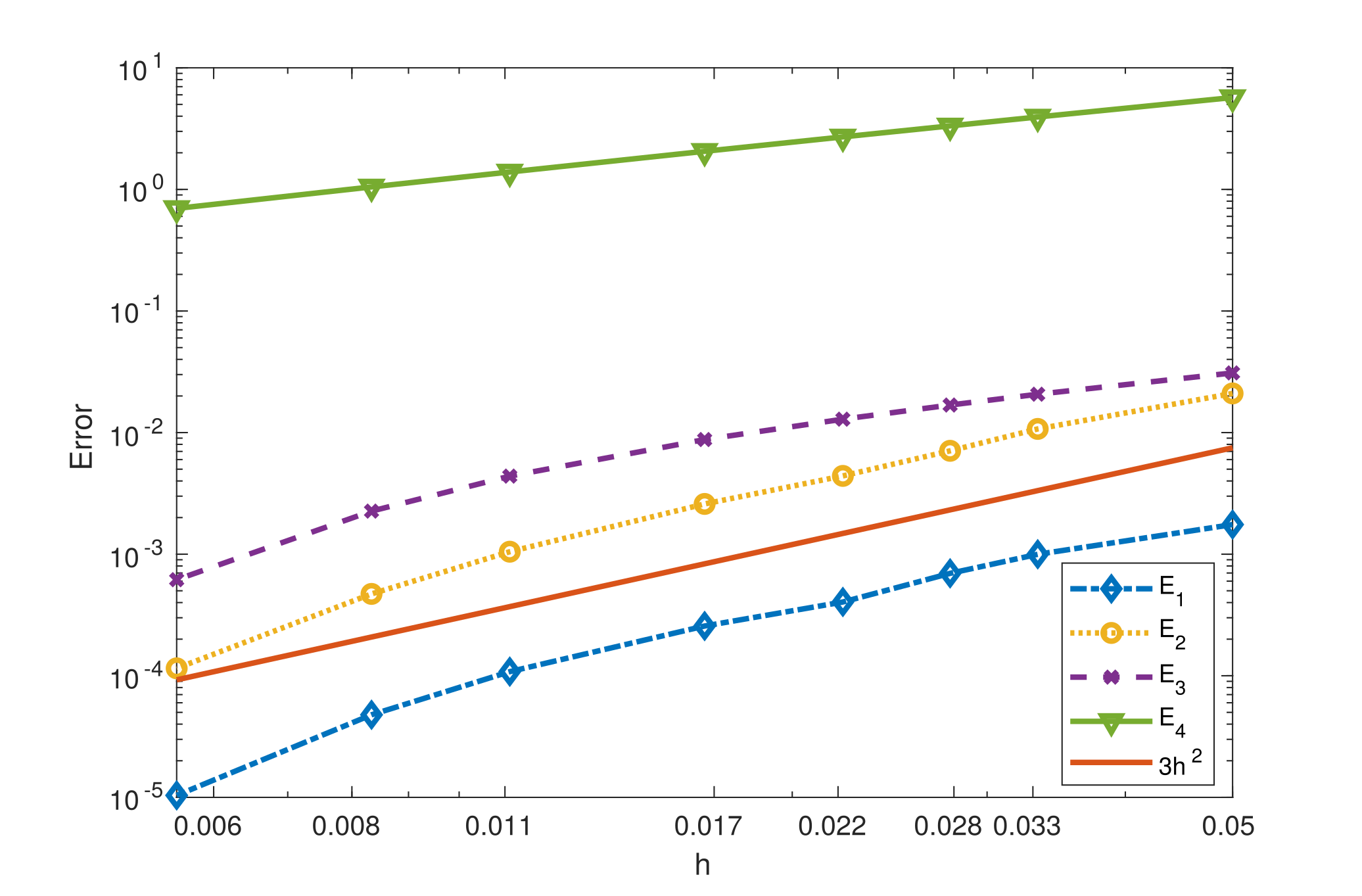} 
\end{center}
\caption{For the convergence test in Subsection \ref{subsec:conv}: $\log-\log$ graph of the errors \eqref{eq:err_defs1}, \eqref{eq:err_defs2} over the spatial step sice $h$ including the graph of $5 h^2$ for comparison.}
\label{fig:conv_err}
\end{figure}

\begin{table}
{\footnotesize
\begin{center}
\begin{tabular}{|r|r||l|l||l|l||l|l||l|l|} \hline 
$J$ & $N$ & $\mathcal{E}_1$ & $\eoc_1$ & $\mathcal{E}_2$ & $\eoc_2$ & $\mathcal{E}_3$ & $\eoc_3$ & $\mathcal{E}_4$ & $\eoc_4$\\ \hline \hline 
 20 & 400 & 0.0017525 & -1 & 0.020997 & -1 & 0.030963 & -1 & 5.6986 & -1 \\ \hline 
 30 & 900 & 0.000998 & 1.446 & 0.010719 & 1.726 & 0.020682 & 1.036 & 3.9424 & 0.946 \\ \hline 
 36 & 1296 & 0.0006957 & 2.039 & 0.0070811 & 2.343 & 0.016851 & 1.158 & 3.331 & 0.953 \\ \hline 
 45 & 2025 & 0.0004045 & 2.490 & 0.0044028 & 2.183 & 0.012891 & 1.231 & 2.7038 & 0.958 \\ \hline 
 60 & 3600 & 0.0002566 & 1.613 & 0.0025894 & 1.881 & 0.0087659 & 1.366 & 2.0585 & 0.966 \\ \hline 
 90 & 8100 & 0.0001081 & 2.160 & 0.0010478 & 2.262 & 0.0043857 & 1.731 & 1.3916 & 0.979 \\ \hline 
 120 & 14400 & 0.0000478 & 2.867 & 0.0004712 & 2.804 & 0.00226 & 2.327 & 1.0485 & 0.994 \\ \hline 
 180 & 32400 & 0.0000104 & 3.790 & 0.0001151 & 3.499 & 0.0006168 & 3.225 & 0.69774 & 1.011 \\ \hline 
\end{tabular}
\end{center}
}
\caption{For the test in Subsection \ref{subsec:conv} (with $\epsilon=10^{-3}$ and $\delta = 0.2 h^2$): Errors \eqref{eq:err_defs1}, \eqref{eq:err_defs2} and $\eocs$.}
\label{tab:conv}
\end{table}

We assess the convergence result of Theorem \ref{th_conv} with a specific example. Let $\tilde{z} := (\sqrt{3}-\sqrt{2})/2$ and 
\[
 \tilde{u}^{(1)}_{0}(x) := 
 \begin{pmatrix}
 \tilde{z} + x(1-\tilde{z}) \\
 (1-\tilde{z}) \tfrac{\sin(\pi x)}{2\pi}
 \end{pmatrix},
 \quad 
 \tilde{u}^{(2)}_{0}(x) := 
 \begin{pmatrix}
 \tilde{z} -\sqrt{3} \tfrac{x}{2} \\
 \sqrt{2} \tfrac{x}{2}
 \end{pmatrix},
 \quad 
 \tilde{u}^{(3)}_{0}(x) := 
 \begin{pmatrix}
 \tilde{z} - \sqrt{3} \tfrac{x^2}{2} \\
 -\sqrt{2} \tfrac{x}{2}
 \end{pmatrix}. 
\]
For the initial triod, these curves were rotated about the origin counter-clockwise by $18^\circ$ to avoid any effects due to alingment with the coordinate axes. Note that the curves meet forming $120^\circ$ angles, and that the end points are on the unit circle. Figure \ref{fig:conv_config} (left) gives an impression of the initial triod. 

We considered the evolution over the time interval $[0,T]$ with $T = 0.2$ and chose $\epsilon = 10^{-3}$. Figure \ref{fig:conv_config} displays a numerical solution at that final time in the middle. We remark that the final configuration is not in equilibrium but continues to evolve to a configuration displayed in Figure \ref{fig:conv_config} on the right, which resembles a Steiner configuration \cite{GP68} consisting of three straight segments.  

We are not aware of any analytical solution satisfying these data and thus numerically computed a reference solution for assessing the convergence. The reference solution is denoted by $\{ \Gamma_{ref,h}^{n_{ref}} \}_{n_{ref}=0}^{N_{ref}}$ where $\Gamma_{ref,h}^{n_{ref}} = (U_{ref}^{(1),n_{ref}},U_{ref}^{(2),n_{ref}},U_{ref}^{(3),n_{ref}})$, and we chose $J_{ref} = 360$ elements and $N_{ref} = 129600$ time steps with corresponding spatial and temporal step sizes denoted by $h_{ref}$ and $\delta_{ref}$, respectively. 

For a computation with discretisation parameters $J$ and $N$ the following errors were computed, where $\mathcal{E}_2(J,N)$ and $\mathcal{E}_3(J,N)$ serve as approximations to the errors in Theorem~\ref{th_conv}: 
\begin{equation} \label{eq:err_defs1} \begin{split} 
 \mathcal{E}_1(J,N) &:= \max_{0 \leq n \leq N} \max_{0 \leq j \leq J} \max_{1 \leq i \leq 3} | U_j^{(i),n} - U_{ref,j_{ref}(j)}^{(i),n_{ref}(n)} |^2, \\
 \mathcal{E}_2(J,N) &:= \max_{0 \leq n \leq N} \sum_{j_{ref}=0}^{J_{ref}-1} \sum_{i=1}^{3} h_{ref} \Big{|} \frac{U_{j(j_{ref})+1}^{(i),n} - U_{j(j_{ref})}^{(i),n}}{h} - \frac{U_{ref,j_{ref}+1}^{(i),n_{ref}(n)} - U_{ref,j_{ref}}^{(i),n_{ref}(n)}}{h_{ref}} \Big{|}^2, \\
 \mathcal{E}_3(J,N) &:= \sum_{n_{ref}=0}^{N_{ref}-1} \delta_{ref} \sum_{i=1}^{3} \int_{I} \Big{|} \frac{U^{(i),n(n_{ref})+1} - U^{(i),n(n_{ref})}}{\delta} - \frac{U_{ref}^{(i),n_{ref}+1} - U_{ref}^{(i),n_{ref}}}{\delta_{ref}} \Big{|}^2 dx. 
\end{split} \end{equation} 
Here, for $n \in \{ 0, \dots, N \}$ given, $n_{ref}(n) \in \{ 0, \dots, N_{ref} \}$ is the index such that $n_{ref} \delta_{ref} = n \delta$ yields the same point in time. Similarly for the spatial index map $j_{ref}(j)$. Inversely, for $n_{ref} \in \{ 0, \dots, N_{ref} \}$ given,  $n(n_{ref}) \in \{ 0, \dots, N \}$ is the index such that $n_{ref} \delta_{ref} \in [n \delta, (n+1) \delta)$, and similarly for the spatial index map $j(j_{ref})$. In the limit as $\epsilon \to 0$, the angles of the analytical solution approach $120^\circ$ (see the discussion in the next subsection around equation \eqref{eq:angle_cond} for more detail). We thus also computed the error of the angles formed at the triple junction: 
\begin{equation} \label{eq:err_defs2} 
 \mathcal{E}_4 := \max_{0 \leq n \leq N} \max_{1 \leq i \leq 3} \Big{|} \angle (\pd_{x} U^{(i {\rm mod } 3+1),n}(0), \pd_{x} U^{((i+1) {\rm mod } 3+1),n}(0)) - 120^\circ \Big{|}.
\end{equation}
Here, we recall that, given two calculations with discretisation parameters $(J_a,N_a)$ and $(J_b, N_b)$, {\it experimental order of convergence} ($\eocs$) for spatial convergence then were computed as 
\begin{equation} \label{eq:EOCs}
 \eoc_i = \frac{\log(\mathcal{E}_i(J_a,N_a)) - \log(\mathcal{E}_{i}(J_b,N_b))}{\log(J_b) - \log(J_a)},
\end{equation}
and analogously for convergence in time with $J$ replaced by $N$ in the denominator. 

In order to assess the convergence in the spatial step size we performed some simulations with differing values of $J$ whilst choosing the time step sizes $\delta = 0.2 h^2$. Table \ref{tab:conv} lists the errors and $\eocs$. Figure \ref{fig:conv_err} displays the errors over the step size $h$. 

The numbers clearly evidence convergence. In Theorem \ref{th_conv} we proved convergence rates of two for $\mathcal{E}_2$ and $\mathcal{E}_3$. This is also what we observe for $\mathcal{E}_{2}$. The results are a bit less conclusive for $\mathcal{E}_{3}$, but its $\eocs$ are well bigger than one, increasing, and finally beyond two. Let us remark that the last simulation with $J=180$ elements has just half the number of elements of the reference solution, which could explain the strong increase of the $\eocs$ for $\mathcal{E}_1$--$\mathcal{E}_3$. For the angles in the triple junction we observe linear convergence of $\mathcal{E}_4$. This seems optimal as (\ref{JPC}) is a condition on the first spatial derivatives and we are using piecewise linear approximations. 

\begin{table}
{\footnotesize
\begin{center}
\begin{tabular}{|r|r||l|l||l|l||l|l||l|l|} \hline 
$J$ & $N$ & $\mathcal{E}_1$ & $\eoc_1$ & $\mathcal{E}_2$ & $\eoc_2$ & $\mathcal{E}_3$ & $\eoc_3$ & $\mathcal{E}_4$ & $\eoc_4$\\ \hline \hline 
 60 & 3456 & 0.0001971 & -- & 0.0024645 & -- & 0.0068268 & -- & 2.0593 & -- \\ \hline 
 60 & 4320 & 0.0001506 & 1.206 & 0.0018466 & 1.294 & 0.005426 & 1.029 & 2.0547 & 0.0101 \\ \hline 
 60 & 5760 & 0.0001014 & 1.375 & 0.0012261 & 1.423 & 0.0038776 & 1.168 & 2.0491 & 0.0096 \\ \hline 
 60 & 6912 & 7.6337e-05 & 1.557 & 0.00092 & 1.576 & 0.0030415 & 1.332 & 2.0456 & 0.0092 \\ \hline 
 60 & 8640 & 5.175e-05 & 1.742 & 0.0006229 & 1.747 & 0.0021686 & 1.516 & 2.0416 & 0.0088 \\ \hline 
 60 & 11520 & 2.8703e-05 & 2.049 & 0.0003460 & 2.044 & 0.0012815 & 1.829 & 2.0368 & 0.0083 \\ \hline 
 60 & 17280 & 9.4174e-06 & 2.749 & 0.0001146 & 2.725 & 0.0004590 & 2.532 & 2.0305 & 0.0076 \\ \hline 
\end{tabular}
\end{center}
}
\caption{For the test in Subsection \ref{subsec:conv} (with $\epsilon=10^{-3}$): Errors \eqref{eq:err_defs1}, \eqref{eq:err_defs2} and $\eocs$ but for $J$ fixed and $N$ changing.}
\label{tab:conv_time}
\end{table}

For completeness, we have also briefly checked the time discretisation error. Fixing $J=60$ we computed a reference solution with $N_{ref} = 34560$ and then compared it with the solutions for several smaller values $N$. Table \ref{tab:conv_time} confirms convergence of $\mathcal{E}_1$--$\mathcal{E}_3$ with $\eocs$ closing in on two (as the errors are squares of norms $\eocs$ of two correspond to linear convergence, which is the expected rate of the first order time stepping scheme). In the last row the rates are well beyond two but this could be due to approaching the resolution of the reference solution, noting that $N=17280$ is half the number of time steps of the reference solution. The angles will only converge if the spatial resolution is improved, and this is visible in terms of stagnating values of $\mathcal{E}_4$.

\subsection{Impact of the regularisation parameter}
\label{subsec:eps}

\begin{figure}
\begin{center}
 \includegraphics[width=5cm]{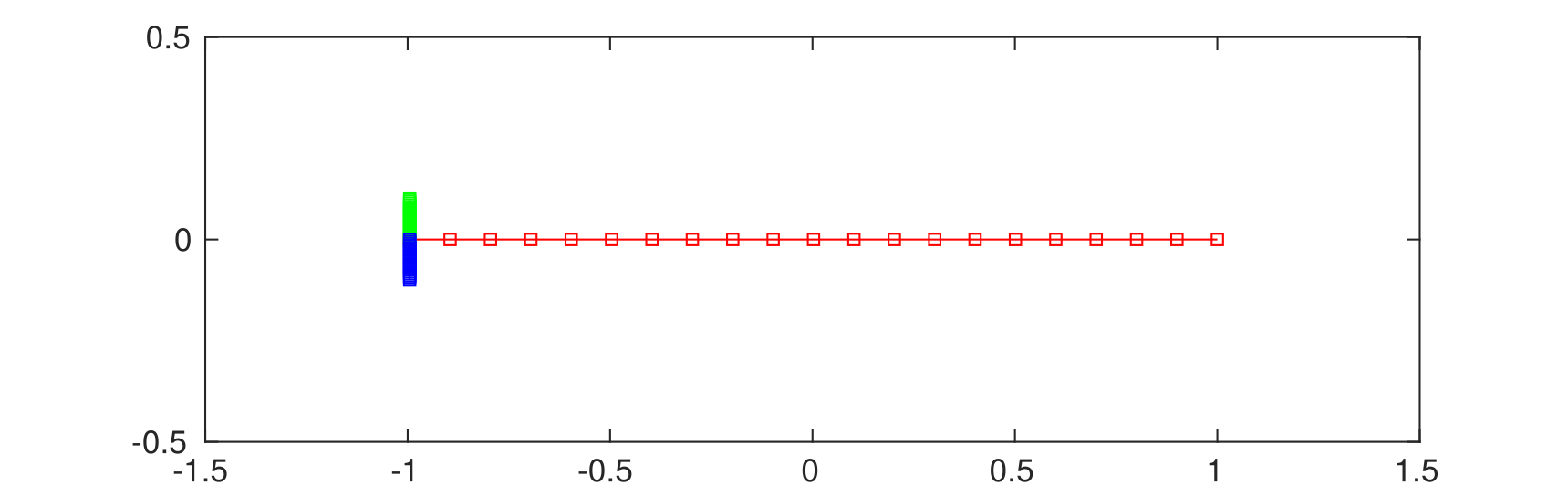} 
 \hfill 
 \includegraphics[width=5cm]{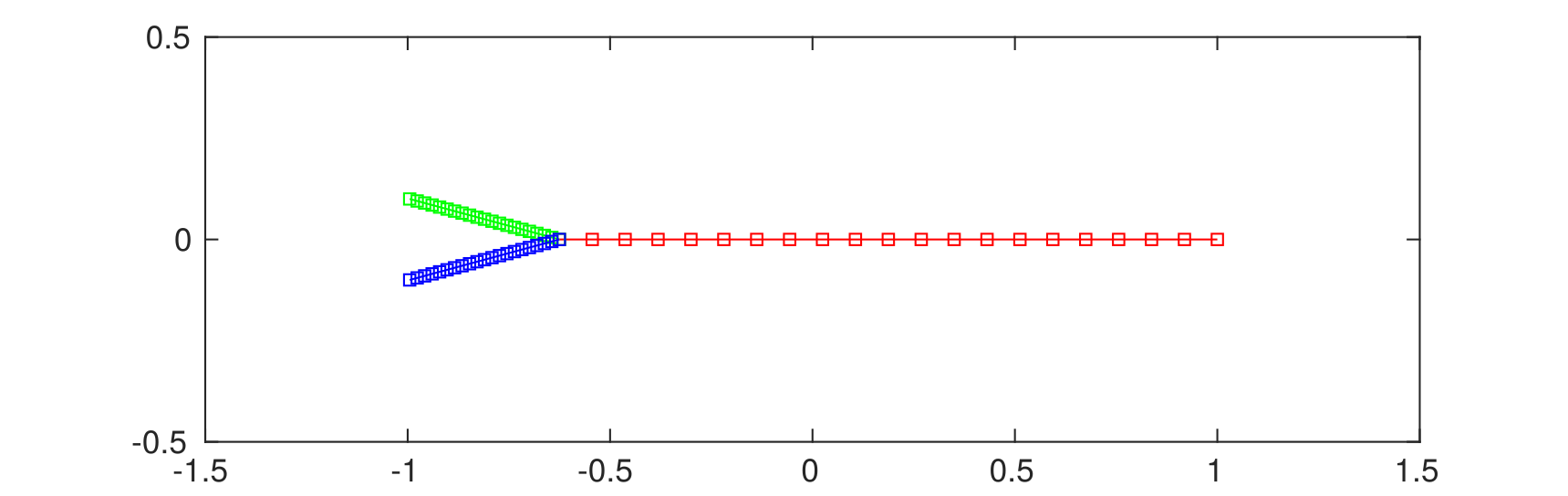} 
 \hfill 
 \includegraphics[width=5cm]{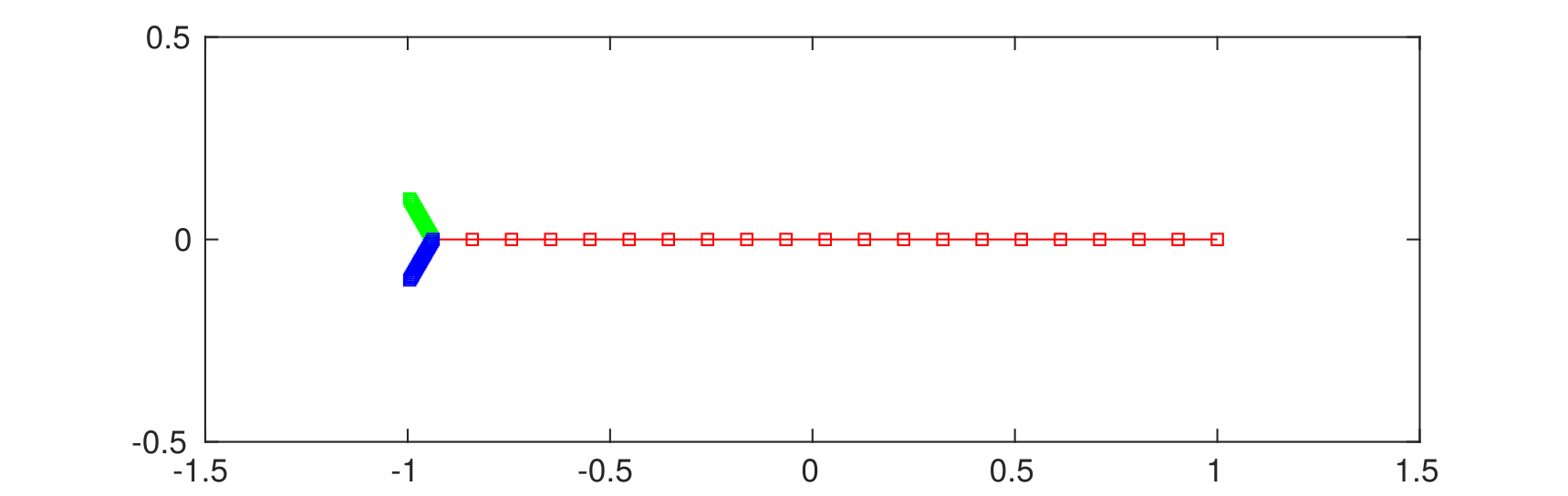} 
\end{center}
\caption{For the test in Subsection \ref{subsec:eps}: For $J=20$, initial configuration (left), and relaxed configurations for $\epsilon = 1$ (centre) and $\epsilon = 10^{-5}$ (right). The curves given by $u^{(1)}$, $u^{(2)}$, and $u^{(3)}$ are red, green, and blue, respectively. The time step size was set to $\delta = 0.01$, and the computation was finished when the stopping criterion \eqref{eq:eps_stopping} was satisfied.}
\label{fig:impact_eps}
\end{figure}

\begin{table}
{\footnotesize
\begin{center}
\begin{tabular}{|r|r|l||l|l||l|l|} \hline 
$J$ & $N_{tot}$ & $\epsilon$ & $\mathcal{E}_{ang}$ & $\eoc_{ang}$ & $\mathcal{E}_{pos}$ & $\eoc_{pos}$ \\ \hline \hline 
 20 & 669 & 1 & 89.719 & -- & 0.31184 & -- \\ \hline 
 20 & 552 & 0.1 & 12.759 & 0.8471 & 0.015937 & 1.2915 \\ \hline 
 20 & 3769 & 0.01 & 1.2665 & 1.0032 & 0.0014832 & 1.0312 \\ \hline 
 20 & 18912 & 0.001 & 0.12656 & 1.0003 & 0.00014736 & 1.0028 \\ \hline 
 20 & 8864 & 0.0001 & 0.012655 & 1.0000 & 1.4726e-05 & 1.0003 \\ \hline 
 20 & 21 & 1e-05 & 0.001264 & 1.0006 & 1.4684e-06 & 1.0012 \\ \hline 
\end{tabular}
\end{center}
}
\caption{For the test in Subsection \ref{subsec:eps}: We display $\mathcal{E}_{ang}$ and $\mathcal{E}_{pos}$ defined in \eqref{eq:eps_err_defs} and corresponding $\eocs$ when varying $\epsilon$ but with $J$ and $\delta$ fixed. The number $N_{tot}$ is the (final) time step when the stopping criterion \eqref{eq:eps_stopping} was satisfied.}
\label{tab:impact_eps}
\end{table}

Recall from \eqref{JPC} the condition
\begin{equation} \label{JPC2}
 0 = \sum_{i=1}^{3} \frac{u^{(i)}_{x}(t,0)}{|u^{(i)}_{x}(t,0)|} + \epsilon u^{(i)}_{x}(t,0) = \sum_{i=1}^{3} \big{(} 1 + \epsilon |u^{(i)}_{x}(t,0)| \big{)} \tau^{(i)}(t,0) =: \sum_{i=1}^{3} \tilde{\sigma}^{(i)} \tau^{(i)}(t,0)
\end{equation}
in the triple junction. Let us denote the angle opposite of the curve defined by $u^{(i)}$ with $\theta^{(i)}$ (see Figure \ref{fig:triod}). Equation \eqref{JPC2} implies that (for instance, see \cite{GarNesSto_SIAP_1999})
\begin{equation} \label{eq:angle_cond}
 \frac{\sin(\theta^{(1)})}{\tilde{\sigma}^{(1)}} = \frac{\sin(\theta^{(2)})}{\tilde{\sigma}^{(2)}} = \frac{\sin(\theta^{(3)})}{\tilde{\sigma}^{(3)}}. 
\end{equation}
In applications, the $\tilde{\sigma}^{(i)}$ can be interpreted as surface tension coefficients, and the higher $\tilde{\sigma}^{(i)}$ the stronger the corresponding curve pulls at the triple junction. If $\epsilon = 0$ then all the $\tilde{\sigma}^{(i)}$ are the same, and this implies $120$ degree angles. But if the length elements $|u^{(i)}_{x}(t,0)|$ differ and $\epsilon$ is positive then we expect to see deviations from these angles. 

We assessed the impact of $\epsilon$ by relaxing the initial curves
\[
 u^{(1)}_{0}(x) := 
 \begin{pmatrix}
 -\tilde{z} + x (1 + \tilde{z}) \\
 0
 \end{pmatrix},
 \quad 
 u^{(2)}_{0}(x) := 
 \begin{pmatrix}
 -\tilde{z} \\
 x z
 \end{pmatrix},
 \quad 
 u^{(3)}_{0}(x) := 
 \begin{pmatrix}
 -\tilde{z} \\
 -x z
 \end{pmatrix}
\]
for $z = 0.1$ and $\tilde{z} = \sqrt{1 - z^2}$ to an equilibrium triod for several values of $\epsilon$. We then compared the angles between the elements forming the triple junction with the $120$ degrees that we would get for $\epsilon = 0$. Note that the initial triod is an inconsistent initial condition in that it does not satisfy the angle condition, but we observed that approximately correct angles emerge very quickly. An equilibrium configuration consists of three straight segments connecting a triple junction on the first coordinate axis to the three (fixed) end points of the initial curve. For $\epsilon = 0$ the position of this final triple junction can be explicitly computed to be $p(0) := (-\tilde{z} + z / \sqrt{3}, 0),$ and we also investigate the impact of $\epsilon$ on the position of the triple junction. 

We performed computations for $J=20$ ($h=0.05$) with a time step size of $\delta = 0.01$. The computations were terminated at the first time step, denoted by $N_{tot}$, such that 
\begin{equation} \label{eq:eps_stopping}
 \max_{1 \leq i \leq 3} \max_{1 \leq j \leq J} \big{|} (U^{(i),N_{tot}}_{j} - U^{(i),N_{tot}-1}_{j}) / \delta \big{|} < 10^{-6}
\end{equation}
was satisfied. Figure \ref{fig:impact_eps} (left) displays the initial configuration and the relaxed configurations for $\epsilon = 1$ (centre) and $\epsilon = 10^{-5}$ (right). The vertices look well equi-distributed for each curve. We also observe that the first curve is much longer than the other two, whence $|u^{(1)}_{hx}| > \max \{ |u^{(2)}_{hx}|, |u^{(3)}_{hx}| \}$. Consequently, $\tilde{\sigma}^{(1)} > \max \{ \tilde{\sigma}^{(2)}, \tilde{\sigma}^{(3)} \}$, and this difference becomes the more pronounced the larger $\epsilon$. For $\epsilon = 1$, Figure \ref{fig:impact_eps}, centre, indeed reveals that the triple junction is positioned significantly further to the right of the position for the limiting problem, i.e., towards the other end point of the curve given by $u^{(1)}_{h}$. 

As mentioned above, we computed the errors defined by
\begin{equation} \label{eq:eps_err_defs}
 \mathcal{E}_{ang}(\epsilon) := \max_{1 \leq i \leq 3} | \theta^{(i)}_{h}(\epsilon) - 120 |, \quad \mathcal{E}_{pos}(\epsilon) := | p_{h}(\epsilon) - p(0) |,
\end{equation}
where $p_{h}(\epsilon) = U^{(1),N_{tot}}(0,\epsilon) = U^{(2),N_{tot}}(0,\epsilon) = U^{(3),N_{tot}}(0,\epsilon)$ is the computed triple junction position with associated angles
\[
 \theta^{(i)}_{h}(\epsilon) = \angle \big{(} \pd_{x} U^{(i {\rm mod } 3+1),N_{tot}}(0,\epsilon), \pd_{x} U^{((i+1) {\rm mod } 3+1),N_{tot}}(0,\epsilon) \big{)}.
\]
The notation is analogous to the continuous case illustrated in Figure \ref{fig:triod}. 
The $\eocs$ were computed analogously to \eqref{eq:EOCs} with $J$ replaced by $1/\epsilon$. 

Table \ref{tab:impact_eps} displays the results. We notice that both errors with respect to the angles and the position converge linearly in $\epsilon$. Further computations (not reported on in detail) showed that the convergence rates don't change significantly when varying the step sizes $h$ and $\delta$. The values for $N_{tot}$ first increase and then decrease again. To some extent this is explained by the fact that the higher $\epsilon$ the further the triple junction moves to the right along the first coordinate axis, see Figure \ref{fig:impact_eps}.

\subsection{Further examples}
\label{subsec:examples}

\begin{figure}
\begin{center}
 \includegraphics[width=5cm]{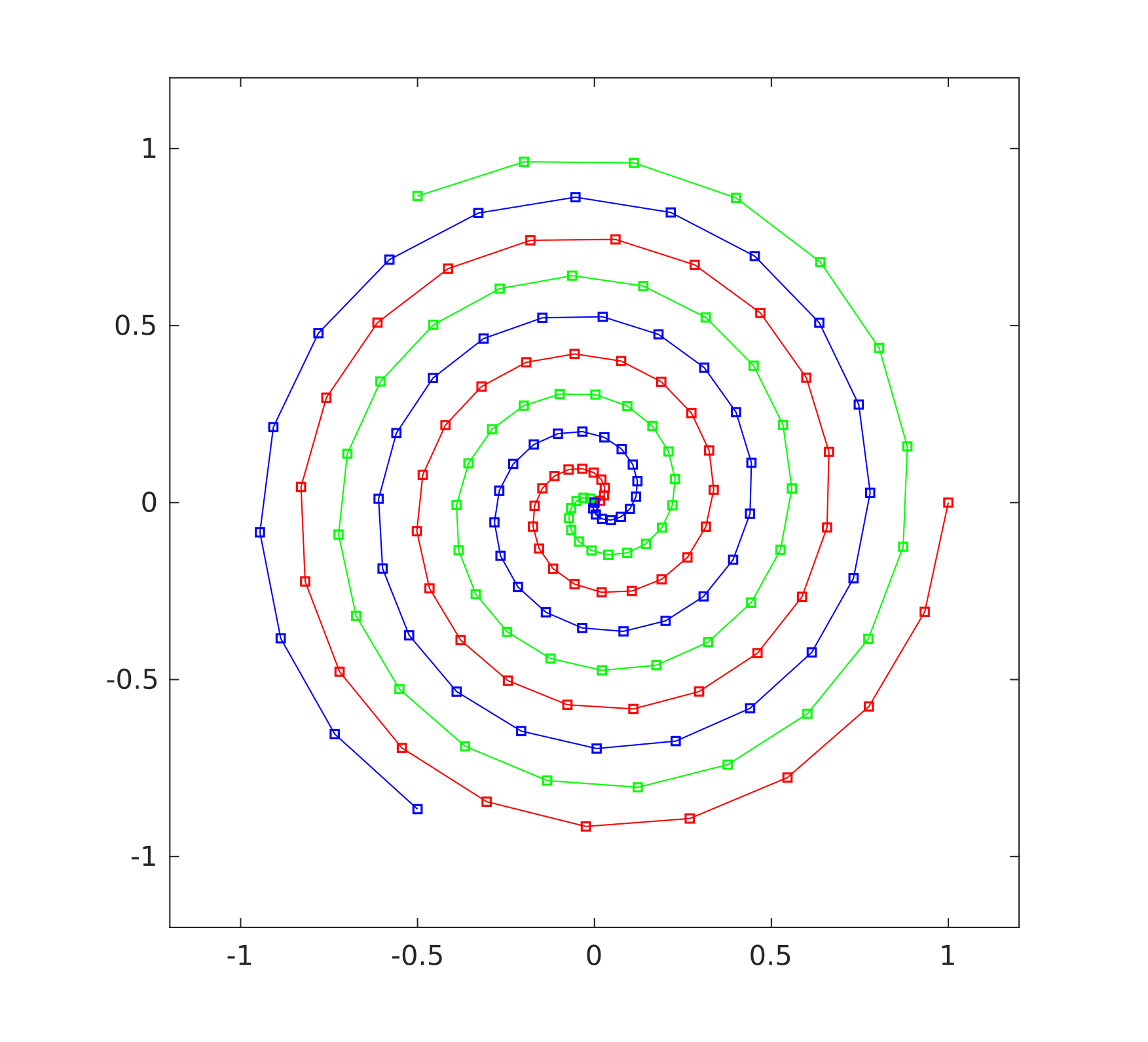} 
 \hfill 
 \includegraphics[width=5cm]{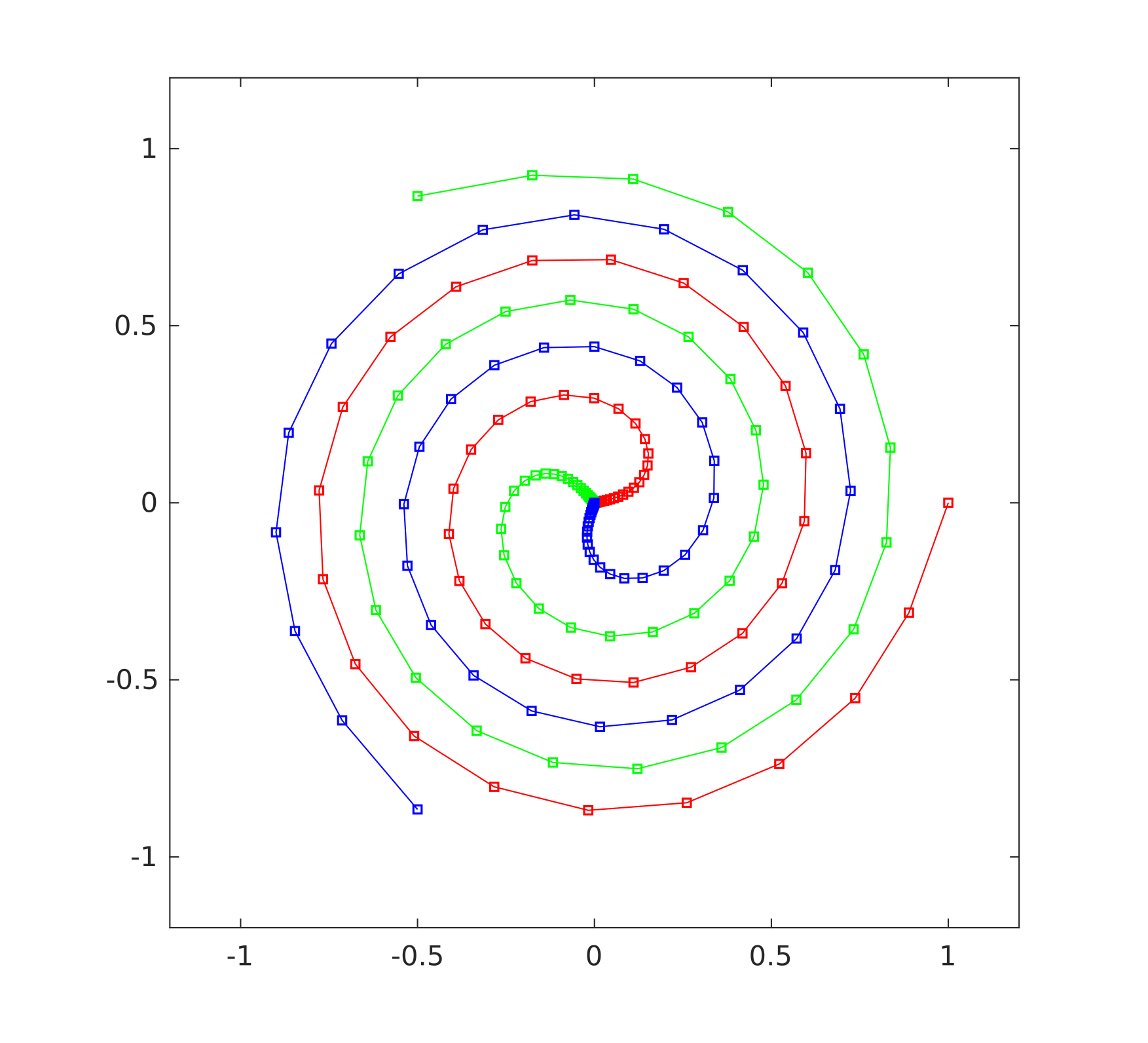} 
 \hfill 
 \includegraphics[width=5cm]{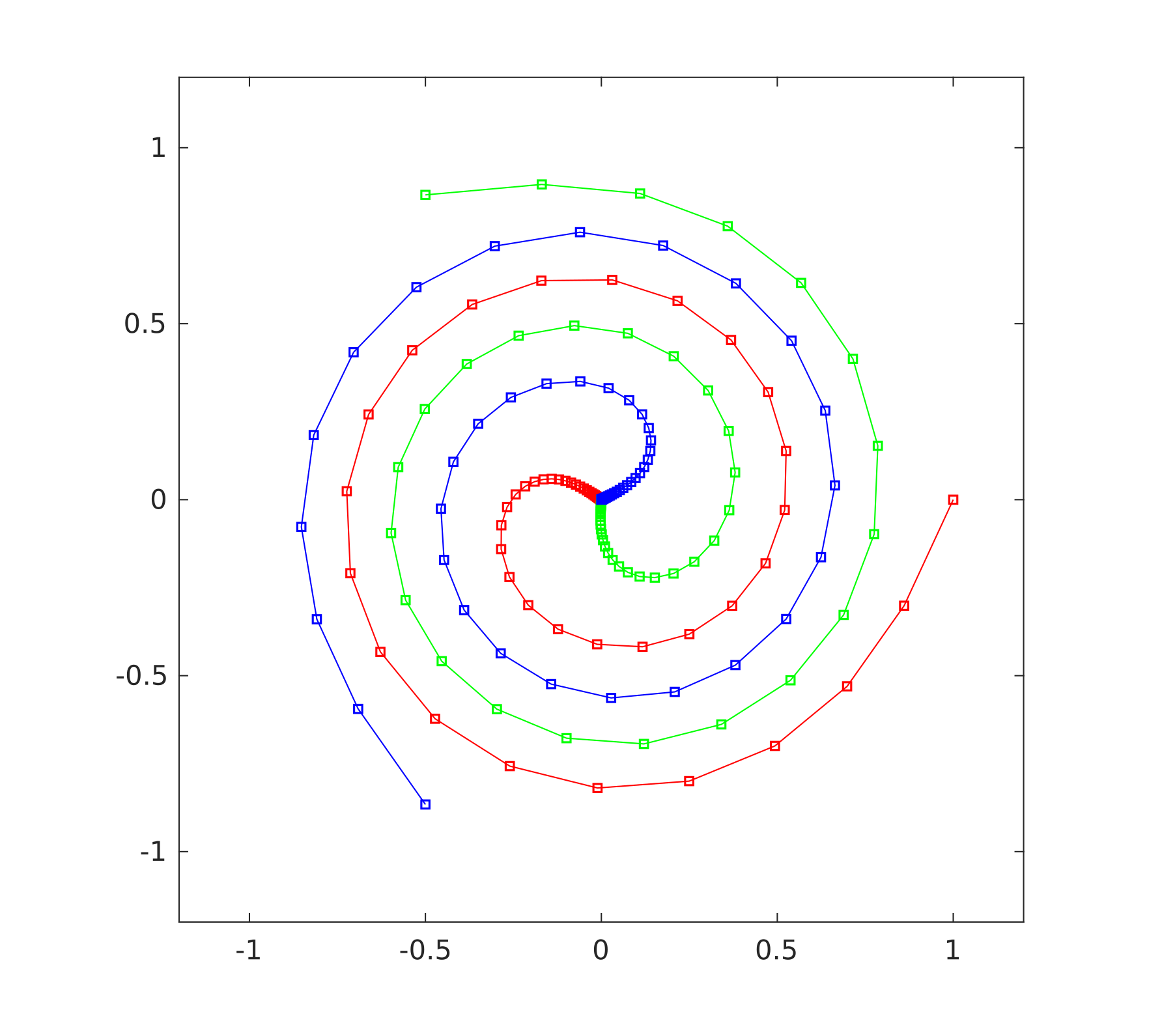} \\
 \includegraphics[width=5cm]{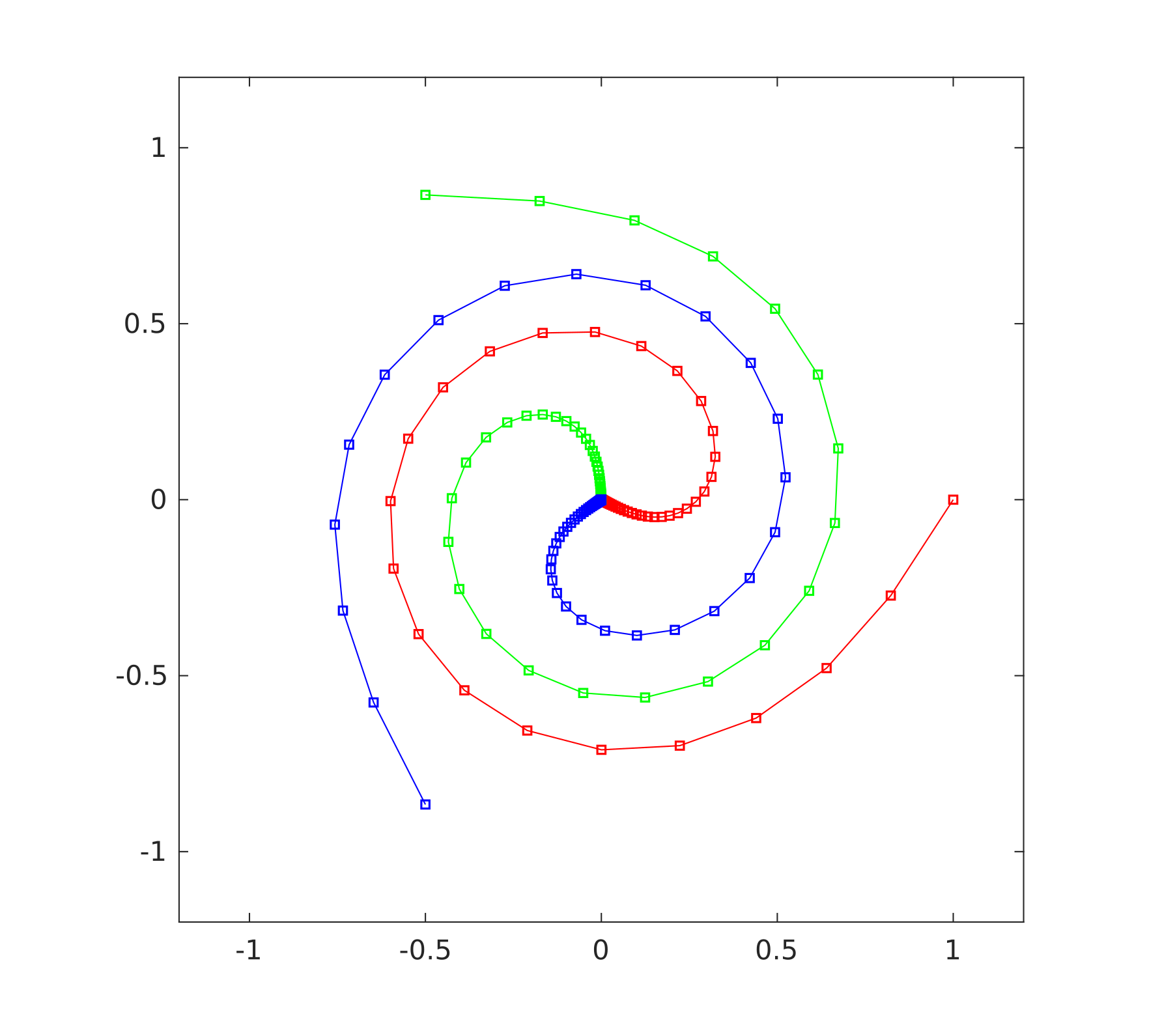} 
 \hfill 
 \includegraphics[width=5cm]{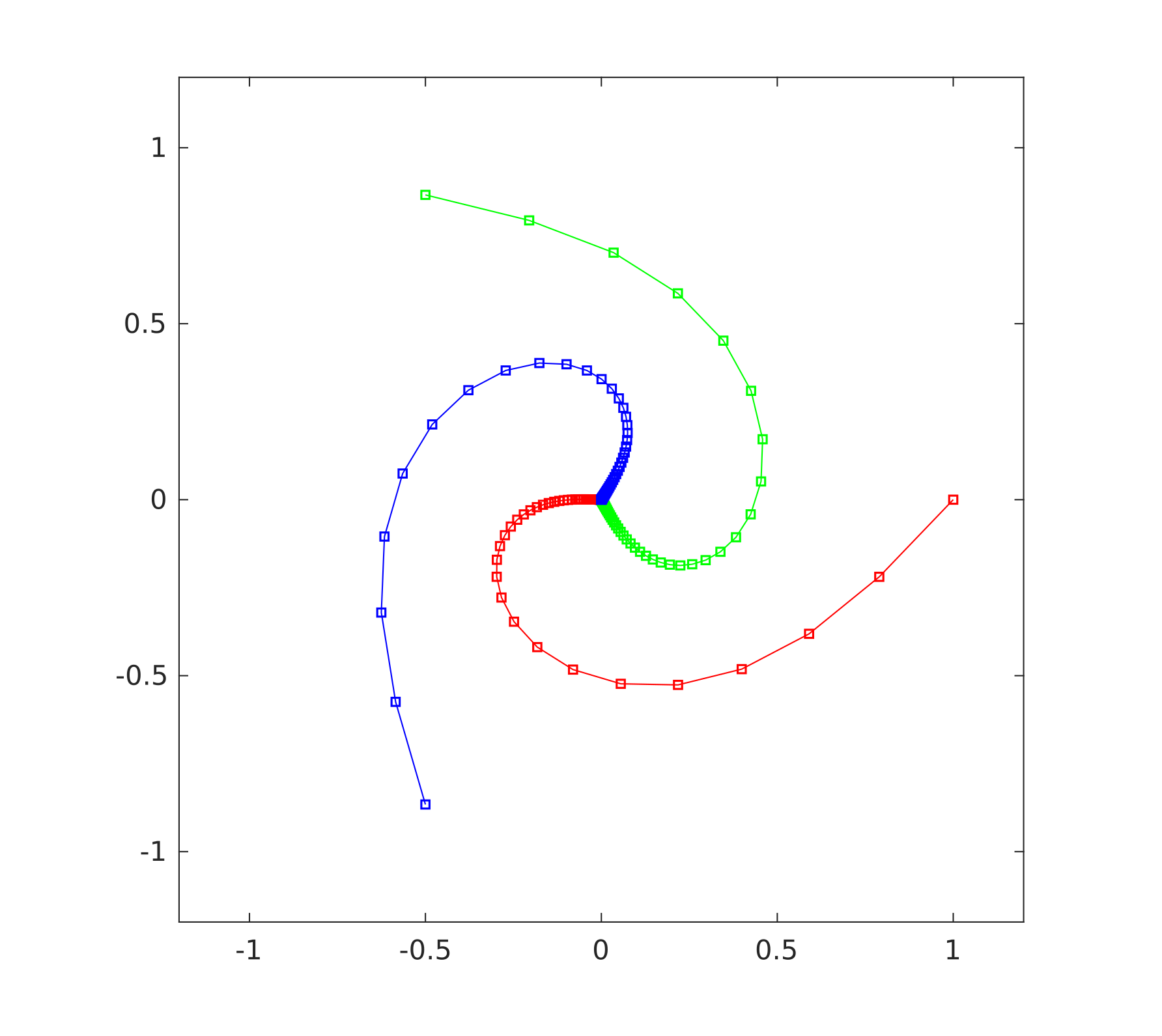} 
 \hfill 
 \includegraphics[width=5cm]{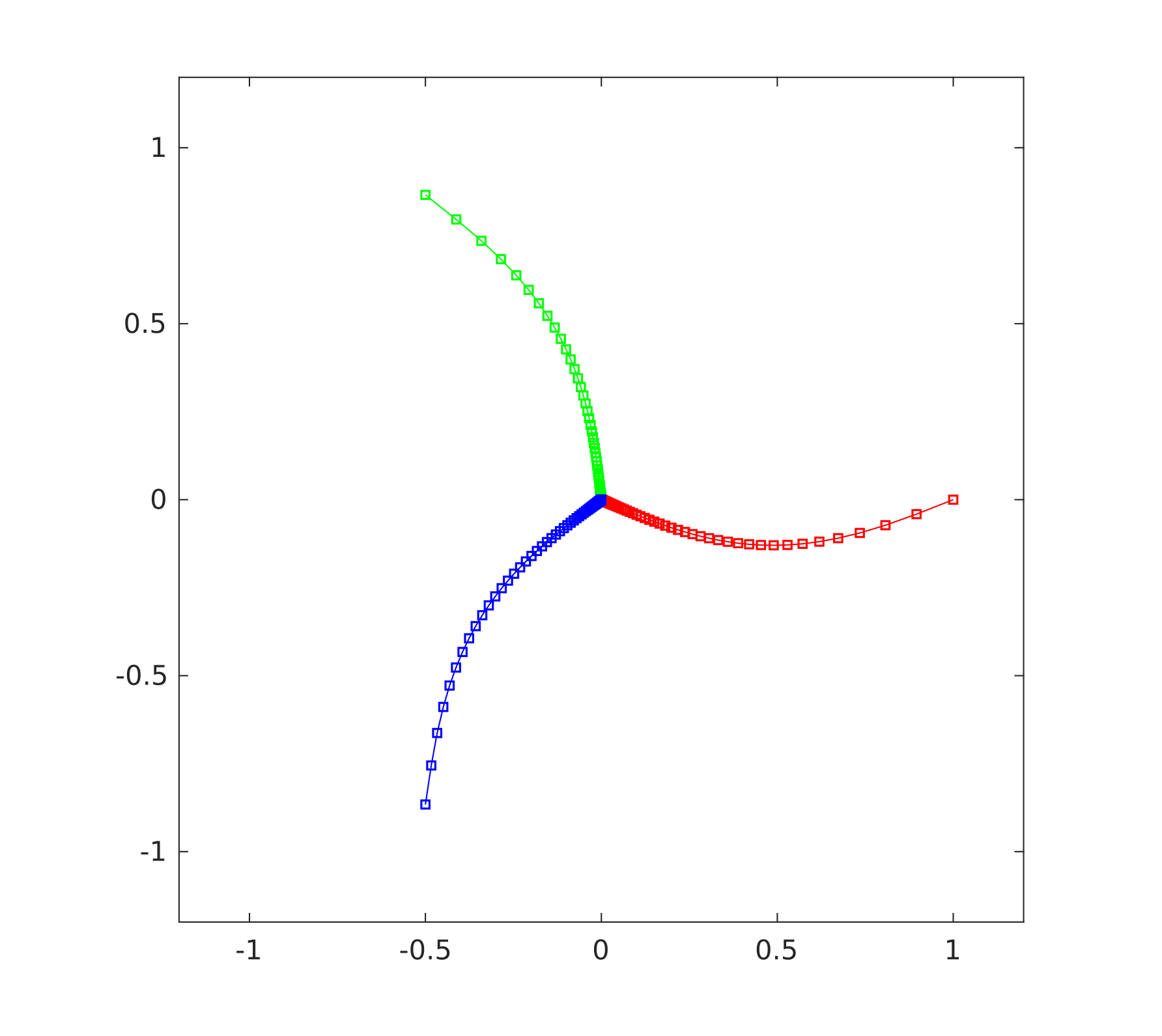} 
\end{center}
\caption{Numerical solution for the initial data given by \eqref{eq:ex_spiral_init} at times $t=0.0, 0.04, 0.08$ (top row, left to right), and $t=0.16, 0.28, 0.48$ (bottom row, left to right). The discretisation parameters were $J=60$ and $\delta = 0.0002$. See Subsection \ref{subsec:examples} for further details.} 
\label{fig:ex_spirals_sol}
\end{figure}

\begin{figure}
\begin{center}
 \includegraphics[width=10cm]{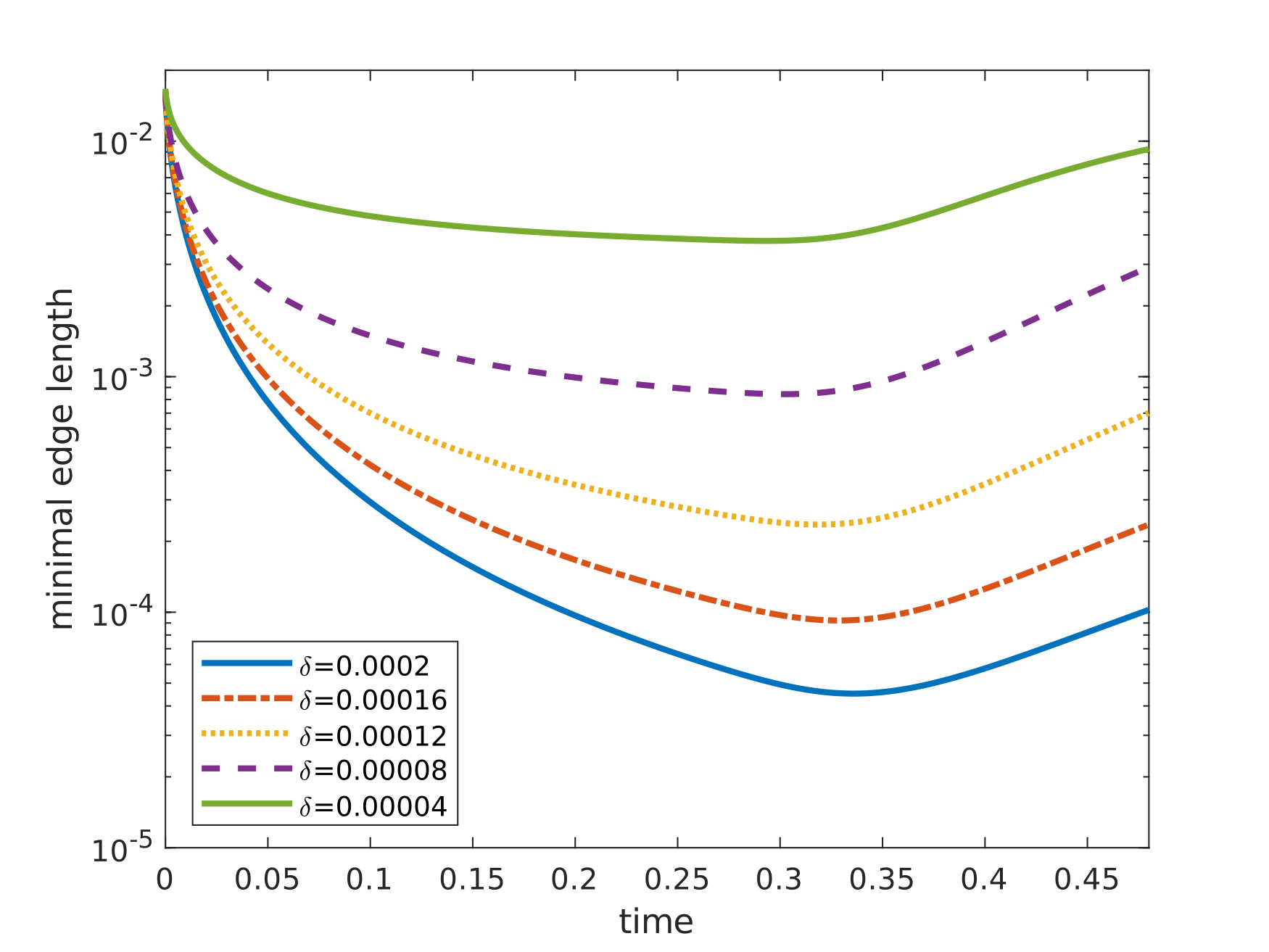} 
\end{center}
\caption{For an example described in Subsection \ref{subsec:examples} with simulations as in Figure \ref{fig:ex_spirals_sol}: Evolution of the minimal segment length for different time step sizes.}
\label{fig:ex_spirals_qvals}
\end{figure}

To assess the capability of the scheme to tangentially redistribute mesh points in the case of strong curvature and, thus, normal velocity we chose some spirals as initial curves, namely 
\begin{equation} \label{eq:ex_spiral_init}
 u^{(i)}_{0}(x) = x \begin{pmatrix} \cos(6 \pi x + \gamma^{(i)}) \\ \sin(6 \pi x + \gamma^{(i)}) \end{pmatrix}, \quad i=1,2,3, 
\end{equation}
with $\gamma^{(0)} = 0$, $\gamma^{(1)} = 2 \pi / 3$, and $\gamma^{(2)} = 4 \pi / 3$. We chose $\epsilon = 10^{-3}$ and set $J=60$. Simulations were run until time $T = 0.48$. 

Figure~\ref{fig:ex_spirals_sol} displays the initial configuration and gives an impression of the numerical solution for the time step size $\delta = 0.0002$. Accumulation of vertices is visible and, usually, becomes worse with increasing time step size. The segments forming the triple junction turned out to be the shortest, and the evolution of their minimum is shown in Figure~\ref{fig:ex_spirals_qvals} for varying time step sizes. We first see a drop, which is the more significant the larger the time step size. But when the triple junction gets closer to equilibrium and the normal velocity becomes smaller then the segment lengths pick up again. A better distribution of mesh points indeed then can be observed. 

\begin{figure}
\begin{center}
 \includegraphics[width=5cm]{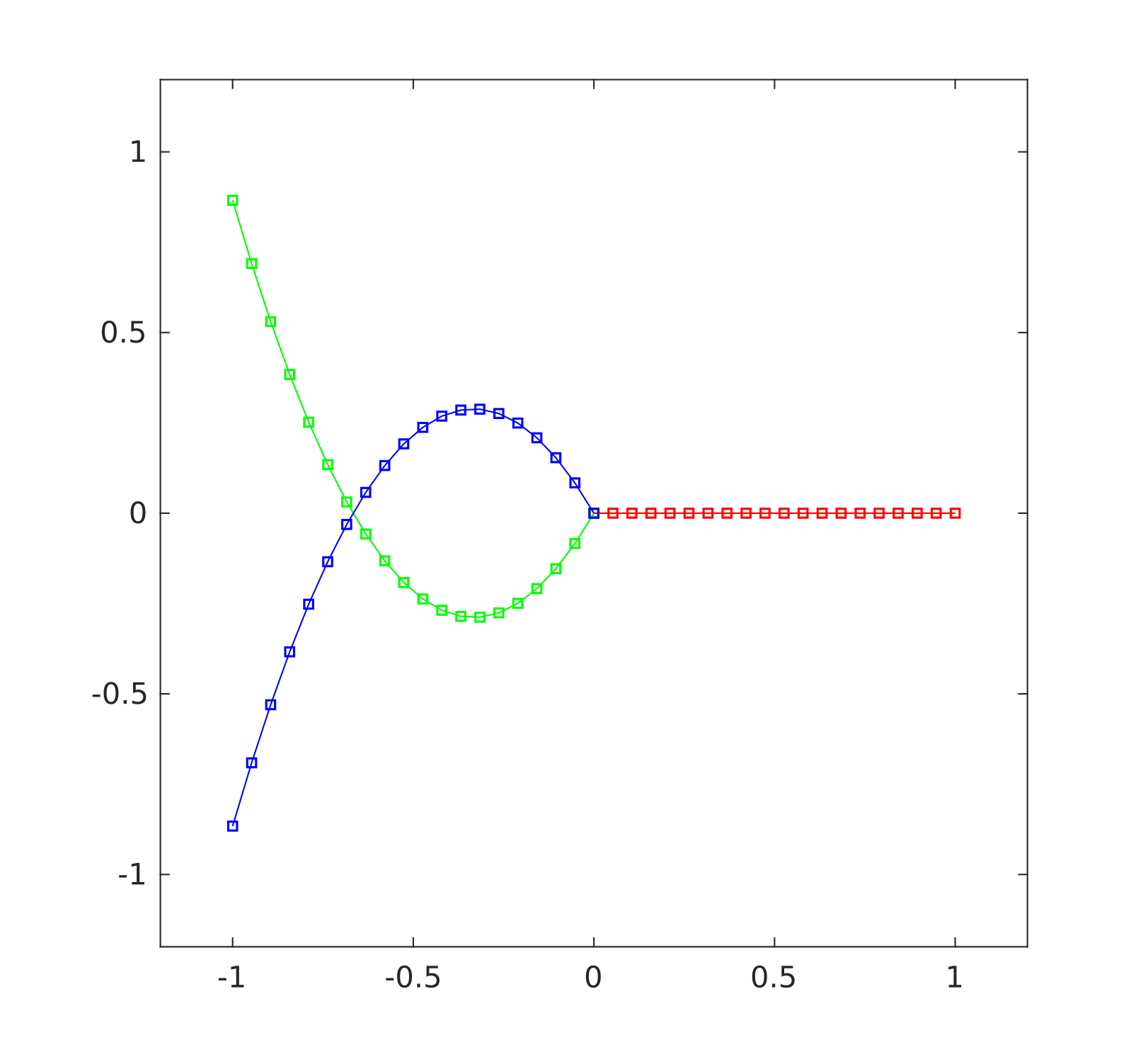} 
 \hfill 
 \includegraphics[width=5cm]{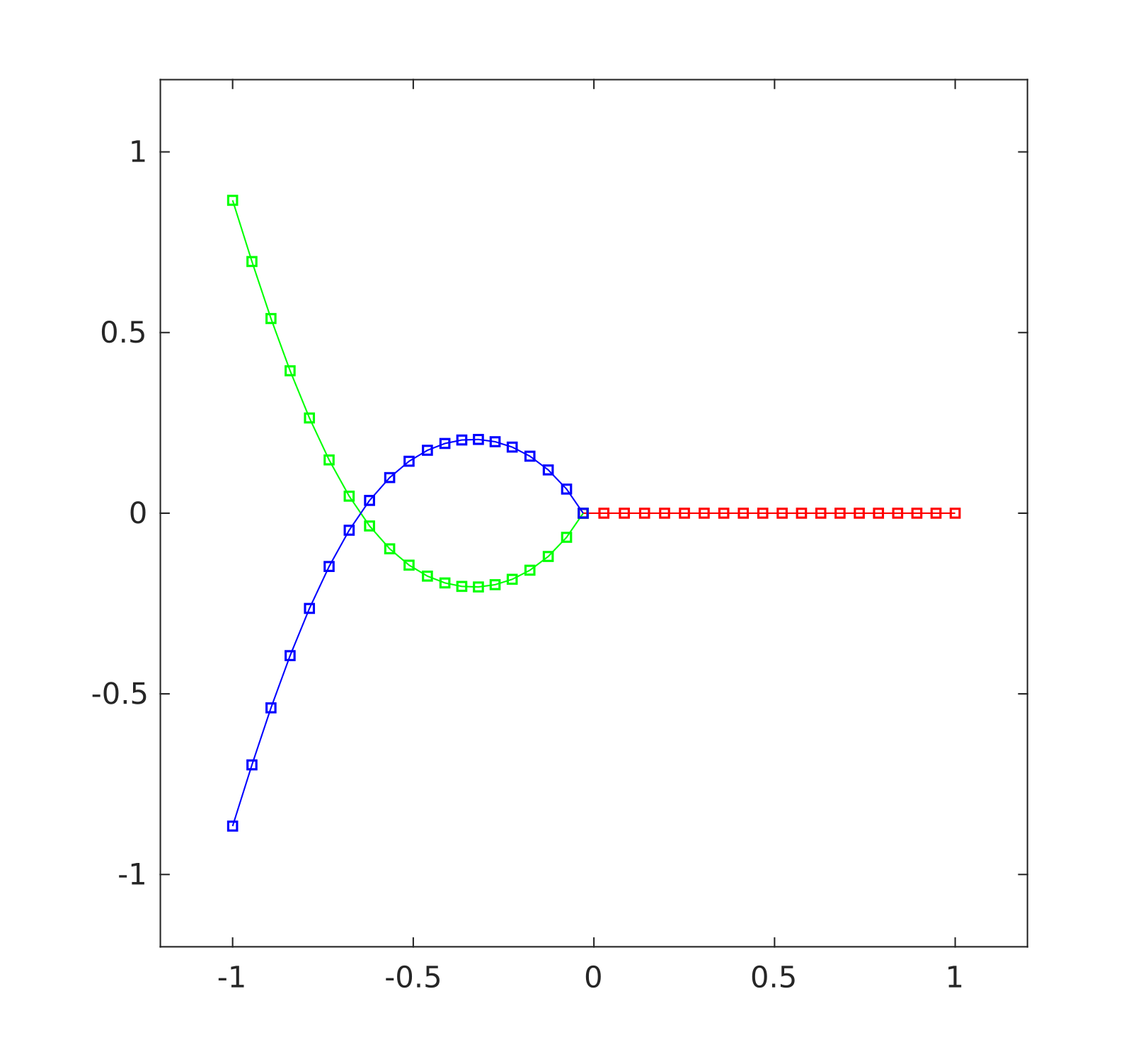} 
 \hfill 
 \includegraphics[width=5cm]{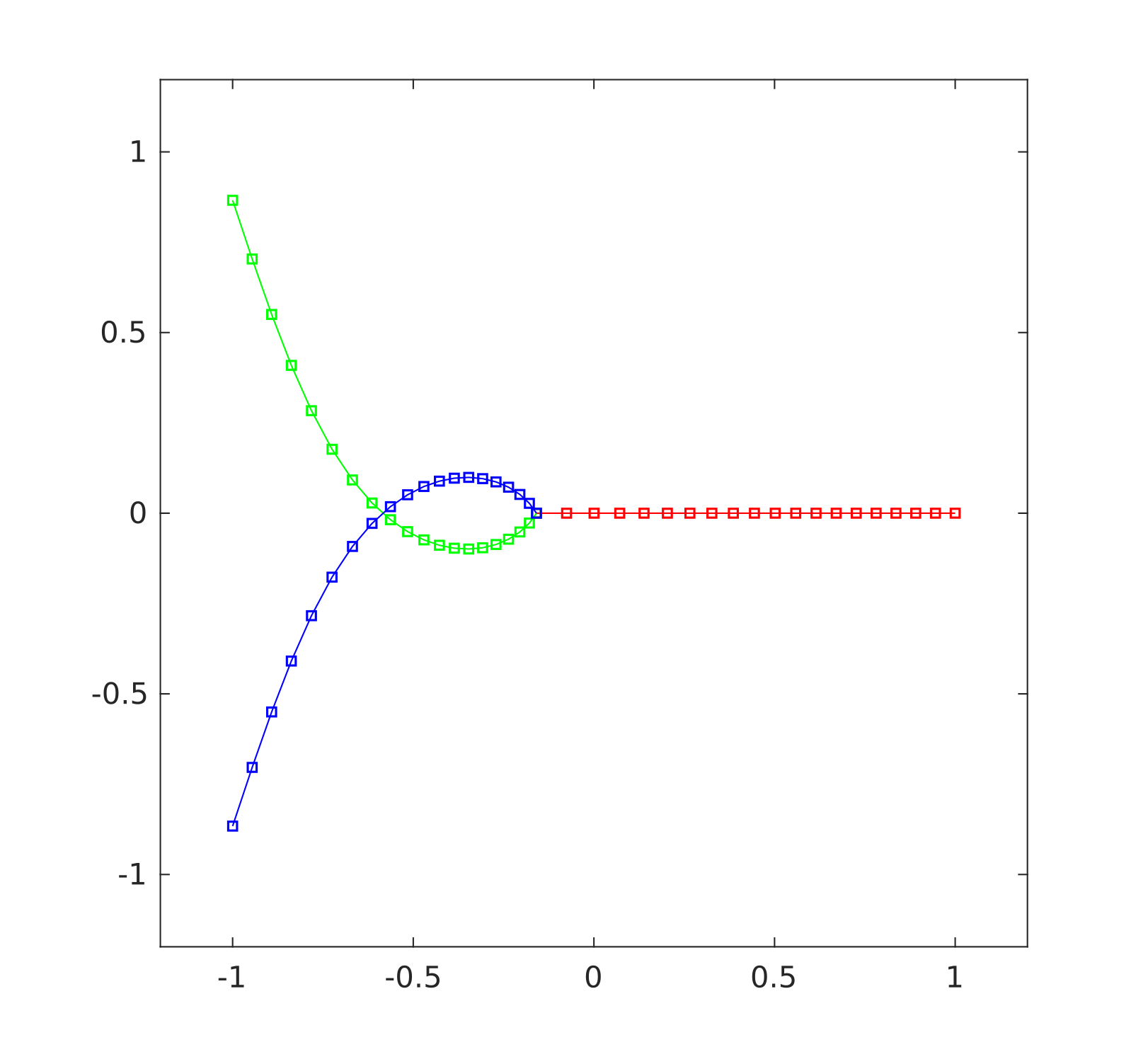} \\
 \includegraphics[width=5cm]{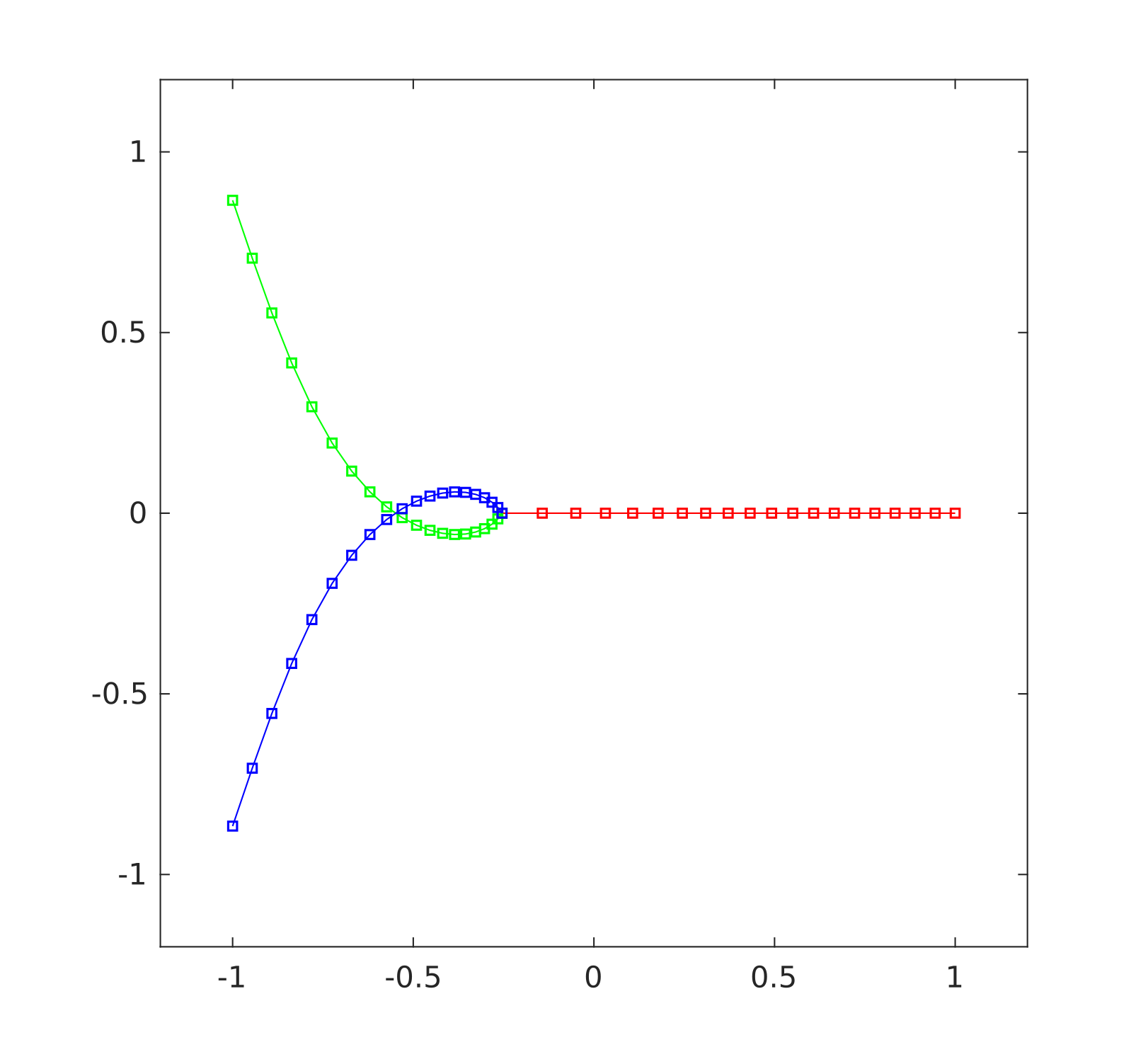} 
 \hfill 
 \includegraphics[width=5cm]{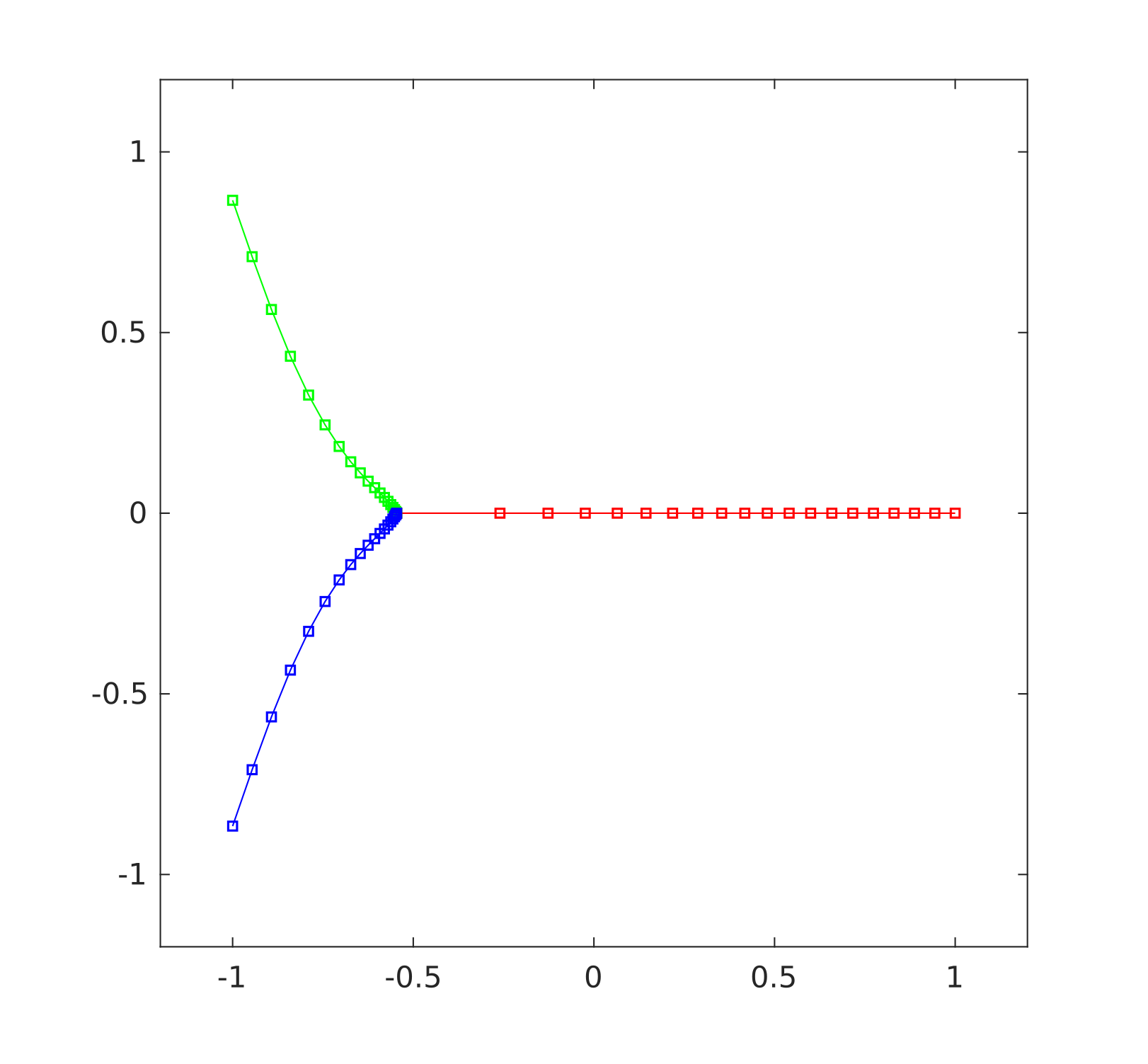} 
 \hfill 
 \includegraphics[width=5cm]{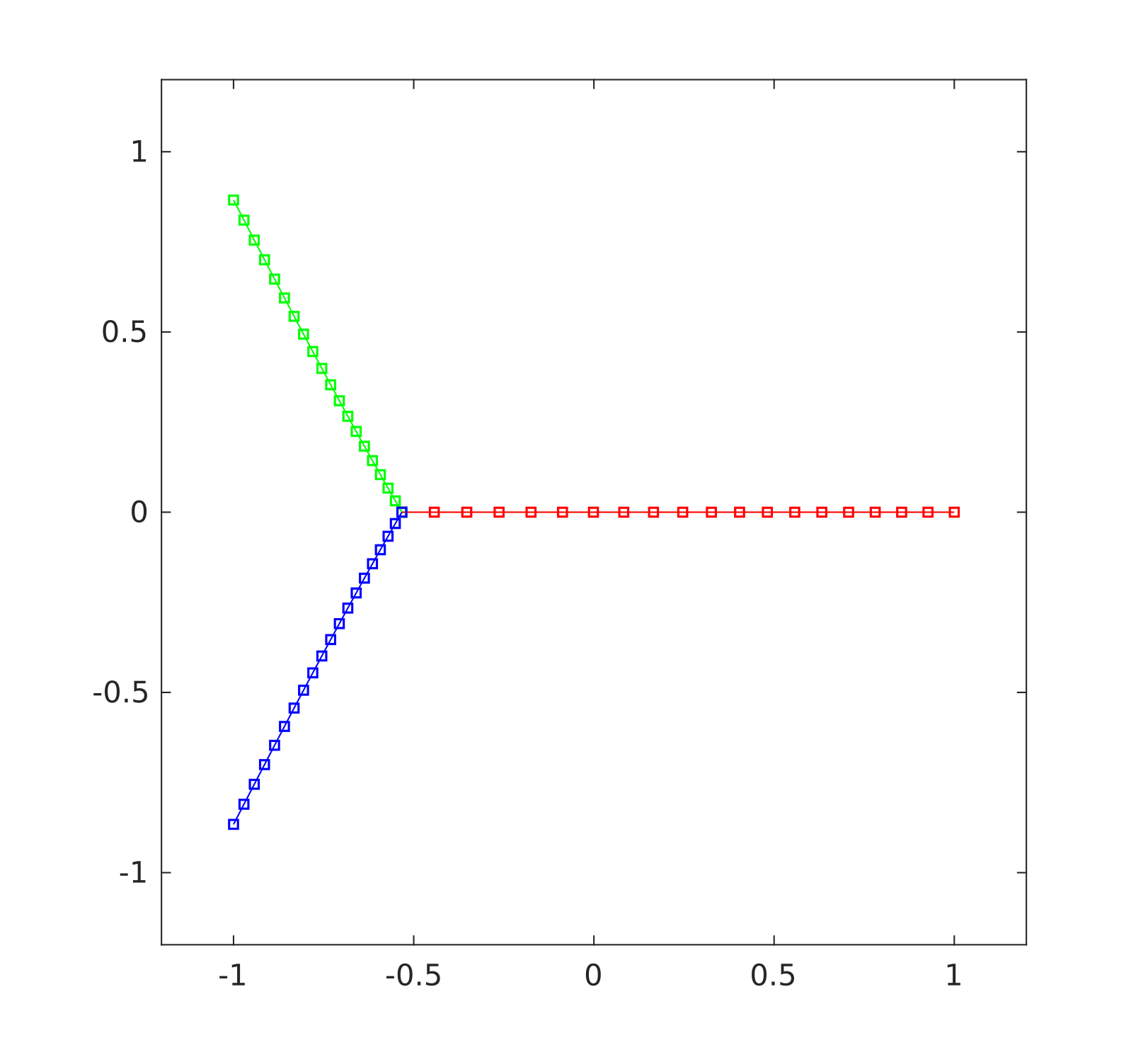} 
\end{center}
\caption{Numerical solution for the initial data given by \eqref{eq:ex_selfintersect_init} at times $t=0.0, 0.02, 0.05$ (top row, left to right), and $t=0.06, 0.07, 0.5$ (bottom row, left to right). The discretisation parameters were $J=20$ and $\delta = 0.0001$. See Subsection \ref{subsec:examples} for further details.} 
\label{fig:ex_selfintersect}
\end{figure}

Fully discrete semi-implicit schemes based on the ideas underpinning \eqref{eq:csf2} to re-distribute vertices are known to jump over singularities caused by self-intersecting curves, see \cite{DecDzi1994} (Figure 1), \cite{BarGarNue2007} (Figure 6), and \cite{EllFri2016} (Figure 8). We were wondering what happens if such self-intersecting curves are connected to a triple junction. We chose the initial data 
\begin{equation} \label{eq:ex_selfintersect_init}
 u^{(1)}_{0}(x) = \begin{pmatrix} x \\ 0 \end{pmatrix}, \quad u^{(2)}_{0}(x) = \begin{pmatrix} -x \\ b(x) \end{pmatrix}, \quad u^{(3)}_{0}(x) = \begin{pmatrix} x \\ -b(x) \end{pmatrix}, \qquad x \in \domain, 
\end{equation}
where $b(x) = \tfrac{3}{2} \sqrt{3} (x - \tfrac{1}{3})^2 - \sqrt{3}/2$. We chose $\epsilon = 10^{-3}$, $J=60$, and $\delta = 10^{-4}$. 

Figure \ref{fig:ex_selfintersect} gives an impression of the initial configuration and the evolution of the numerical solution. Between times $t=0.06$ and $t=0.07$ the topology changes and the self-intersection is lost. The scheme continues to relax the triod towards an equilibrium configuration. Note that the continuous problem develops a singularity so that Assumption~\ref{ass_cont_sol} is not satisfied and our theoretical result doesn't apply. The velocity becomes large around the topological change, which manifests by accumulation of vertices ($U^{(2),n}$, green, and $U^{(3),n}$, blue in Figure~\ref{fig:ex_selfintersect}) and streching of segments elsewhere ($U^{(1),n}$, red). Tangential re-distribution of vertices takes place at a slower pace after, which is visible comparing the last to images of Figure~\ref{fig:ex_selfintersect}. Whilst jumping over such singularities might be desired in some applications, detecting and accurately simulating them might be desired in others. This is likely to require adaptive time stepping and is left for future investigations.

%%%%%%%%%%%%%%%%%%%%%%%%%%%%%%%%%%%%%%%%%%%


\begin{thebibliography}{10}

\bibitem{BalMik2011}
{\sc Bala{\v z}ovjech, M., and Mikula, K.}
\newblock A higher order scheme for a tangentially stabilized plane curve
  shortening flow with a driving force.
\newblock {\em SIAM Journal on Scientific Computing 33}, 5 (2011), 2277--2294.

\bibitem{BarGarNue2007}
{\sc Barrett, J.~W., Garcke, H., and N{\"u}rnberg, R.}
\newblock On the variational approximation of combined second and fourth order
  geometric evolution equations.
\newblock {\em SIAM Journal on Scientific Computing 29}, 3 (Jan. 2007),
  1006--1041.

\bibitem{BarGarNue2011}
{\sc Barrett, J.~W., Garcke, H., and N{\"u}rnberg, R.}
\newblock The approximation of planar curve evolutions by stable fully implicit
  finite element schemes that equidistribute.
\newblock {\em Numerical Methods for Partial Differential Equations 27}, 1
  (2011), 1--30.

\bibitem{BreMas2017}
{\sc Bretin, E., and Masnou, S.}
\newblock A new phase field model for inhomogeneous minimal partitions, and
  applications to droplets dynamics.
\newblock {\em Interfaces and Free Boundaries 19}, 2 (2017), 141--182.

\bibitem{BroGarSto1998}
{\sc Bronsard, L., Garcke, H., and Stoth, B.}
\newblock A multi-phase {{Mullins}}-{{Sekerka}} system: {{Matched}} asymptotic
  expansions and an implicit time discretisation for the geometric evolution
  problem.
\newblock {\em Proceedings of the Royal Society of Edinburgh Section A:
  Mathematics 128}, 3 (1998), 481--506.

\bibitem{BroRei1993}
{\sc Bronsard, L., and Reitich, F.}
\newblock On three-phase boundary motion and the singular limit of a
  vector-valued {{Ginzburg}}-{{Landau}} equation.
\newblock {\em Archive for Rational Mechanics and Analysis 124}, 4 (1993),
  355--379.

\bibitem{BroWet1993}
{\sc Bronsard, L., and Wetton, B.~T.}
\newblock A numerical method for tracking curve networks moving with curvature
  motion.
\newblock {\em Journal of Computational Physics 120}, 1 (1993), 66--87.

\bibitem{Cah1977}
{\sc Cahn, J.~W.}
\newblock Critical point wetting.
\newblock {\em The Journal of Chemical Physics 66}, 8 (1977), 3667--3672.

\bibitem{DecDzi1994}
{\sc Deckelnick, K., and Dziuk, G.}
\newblock On the approximation of the curve shortening flow.
\newblock In {\em Calculus of {{Variations}}, {{Applications}} and
  {{Computations}}: {{Pont}}-{\`a}- {{Mousson}}, 1994}. {Pitman Research Notes
  in Mathematics Series}, 1994, pp.~100--108.

\bibitem{DecDziEll2005}
{\sc Deckelnick, K., Dziuk, G., and Elliott, C.~M.}
\newblock Computation of geometric partial differential equations and mean
  curvature flow.
\newblock {\em Acta Numerica 14\/} (2005), 139--232.

\bibitem{Dzi1991}
{\sc Dziuk, G.}
\newblock An algorithm for evolutionary surfaces.
\newblock {\em Numerische Mathematik 58}, 1 (1991), 603--611.

\bibitem{Dzi1994}
{\sc Dziuk, G.}
\newblock Convergence of a semi-discrete scheme for the curve shortening flow.
\newblock {\em Mathematical Models and Methods in Applied Sciences 4}, 4
  (1994), 589--606.

\bibitem{EllFri2016}
{\sc Elliott, C.~M., and Fritz, H.}
\newblock On approximations of the curve shortening flow and of the mean
  curvature flow based on the {{DeTurck}} trick.
\newblock {\em IMA Journal of Numerical Analysis\/} (June 2016), drw020.

\bibitem{GarNesSto_SIAP_1999}
{\sc Garcke, H., Nestler, B., and Stoth, B.}
\newblock A multiphase field concept: {{Numerical}} simulations of moving phase
  boundaries and multiple junctions.
\newblock {\em SIAM Journal on Applied Mathematics 60}, 1 (Jan. 1999),
  295--315.

\bibitem{GP68}
{\sc Gilbert, E.~N., and Pollak, H.~O.}
\newblock Steiner minimal trees.
\newblock {\em SIAM Journal on Applied Mathematics 16}, 1 (1968), 1--29.

\bibitem{Her1951}
{\sc Herring, C.}
\newblock Surface diffusion as a motivation for sintering.
\newblock In {\em The Physics of Powder Metallurgy}. {McGraw-Hill, New York},
  1951, pp.~143--179.

\bibitem{MacNolRowIns2019}
{\sc Mackenzie, J.~A., Nolan, M., Rowlatt, C.~F., and Insall, R.~H.}
\newblock An adaptive moving mesh method for forced curve shortening flow.
\newblock {\em SIAM Journal on Scientific Computing 41}, 2 (Jan. 2019),
  A1170--A1200.

\bibitem{ManNovPluSch2016}
{\sc Mantegazza, C., Novaga, M., Pluda, A., and Schulze, F.}
\newblock Evolution of networks with multiple junctions.
\newblock {\em arXiv:1611.08254 [math]\/} (Nov. 2016).

\bibitem{MerBenOsh1994}
{\sc Merriman, B., Bence, J.~K., and Osher, S.~J.}
\newblock Motion of multiple junctions: {{A}} level set approach.
\newblock {\em Journal of Computational Physics 112}, 2 (1994), 334--363.

\bibitem{MikRemSarSev2014}
{\sc Mikula, K., Reme{\v s}{\'i}kov{\'a}, M., Sarkoci, P., and {\v S}ev{\v
  c}ovi{\v c}, D.}
\newblock Manifold evolution with tangential redistribution of points.
\newblock {\em SIAM Journal on Scientific Computing 36}, 4 (Jan. 2014),
  A1384--A1414.

\bibitem{PriYu1994}
{\sc Priester, L., and Yu, D.}
\newblock Triple junctions at the mesoscopic, microscopic and nanoscopic
  scales.
\newblock {\em Materials Science and Engineering: A 188}, 1-2 (1994), 113--119.

\bibitem{SmiSolCho2002}
{\sc Smith, K.~A., Solis, F.~J., and Chopp, D.~L.}
\newblock A projection method for motion of triple junctions by level sets.
\newblock {\em Interfaces and Free Boundaries 4}, 3 (2002), 263--276.

\bibitem{TayCahHan1992}
{\sc Taylor, J.~E., Cahn, J.~W., and Handwerker, C.~A.}
\newblock Geometric models of crystal growth.
\newblock {\em Acta Metallurgica et Materialia 40}, 7 (1992), 1443--1474.

\bibitem{Teschl2012ODEs}
{\sc Teschl, G.}
\newblock Ordinary differential equations and dynamical systems.
\newblock {\em American Mathematical Society 140\/} (2012), 364.

\bibitem{You1805}
{\sc Young, T.}
\newblock An essay on the cohesion of fluids.
\newblock {\em Philosophical Transactions of the Royal Society of London 95\/}
  (11805), 65--87.

\end{thebibliography}
\end{document}